\newif\ifconfver
\confverfalse      

\newif\ifcutshort      
\cutshorttrue

\newif\ifcutshortlvltwo  
\cutshortlvltwofalse

\ifconfver
\documentclass[10pt,twocolumn,twoside]{IEEEtran}
\else
\documentclass[11pt, draftclsnofoot, onecolumn]{IEEEtran}
\fi

\usepackage{subfigure}
\usepackage{blkarray}
\usepackage{bigdelim}
\usepackage[ruled]{algorithm2e}
\usepackage{enumerate}

\usepackage{pifont}
\usepackage[dvipsnames]{xcolor}
\usepackage{cite}
\usepackage{float}
\usepackage{threeparttable}

\usepackage{amsmath,amsfonts,bm,amssymb,epsfig,psfrag,ifthen,color}
\usepackage{booktabs}
\usepackage{multirow}
\usepackage{array}
\usepackage{calc}
\usepackage{graphicx}
\usepackage{psfrag}

\usepackage{stfloats}
\usepackage{url}
\usepackage[caption=false]{subfig}
\usepackage{longtable}
\usepackage{amsmath,amsfonts,amssymb,amsbsy,bm,paralist,theorem,ifthen,color}
\usepackage{mathtools}

\usepackage{pifont}       
\usepackage{bbding}       
\usepackage{fontawesome}

\ifconfver

\else

\fi

\def\VectorFont{\bf}
\newcommand{\vx}{{\VectorFont x}}
\newcommand{\vy}{{\VectorFont y}}
\newcommand{\vc}{{\VectorFont c}}
\newcommand{\vu}{{\VectorFont u}}
\newcommand{\vz}{{\VectorFont z}}
\newcommand{\vw}{{\VectorFont w}}
\newcommand{\vlambda}{{\boldsymbol{\lambda}}}

\newtheorem{definition}{Definition}%
\newtheorem{assumption}{Assumption}
\newtheorem{remark}{Remark}%
\newtheorem{theorem}{Theorem}
\newtheorem{lemma}{Lemma}
\usepackage{caption}

\newtheorem{proof}{Proof}

\ifconfver

\else

\fi

\begin{document}

\bibliographystyle{IEEEtran}

\title{ {An Accelerated Stochastic ADMM for
Nonconvex and Nonsmooth Finite-Sum Optimization}}

\ifconfver \else {\linespread{1.1} \rm \fi

\author{\vspace{0.8cm}Yuxuan~Zeng,
~Zhiguo~Wang,
~Jianchao~Bai, 
~and Xiaojing~Shen\\
\thanks{
Yuxuan Zeng, Zhiguo Wang  (corresponding author), and Xiaojing Shen are with College of Mathematics, Sichuan University, 610064, Chengdu, China (e-mail: 2020222010085@stu.scu.edu.cn,~wangzhiguo@scu.edu.cn,~shenxj@scu.edu.cn).}
\thanks{Jianchao Bai  (corresponding author) is with Research \& Development Institute of Northwestern Polytechnical University in Shenzhen, Shenzhen 518057, China;  School of Mathematics and Statistics,  Northwestern Polytechnical
University, Xi'an  710129,      China (e-mail: jianchaobai@nwpu.edu.cn).}
}

\maketitle

\vspace{-1.5cm}
\begin{center}
\today
\end{center}\vspace{0.5cm}

\begin{abstract}
The nonconvex and nonsmooth finite-sum optimization problem with linear constraint has attracted much attention in the fields of artificial intelligence, computer, and  mathematics, due to its wide applications in machine learning and the lack  of efficient algorithms with convincing convergence theories. 
A popular approach to solve it is the  stochastic Alternating Direction Method of Multipliers (ADMM), but most stochastic ADMM-type methods focus on convex models. In addition, the variance reduction (VR) and acceleration techniques are useful tools in the development of stochastic methods due to their simplicity and practicability in providing acceleration characteristics of various machine learning models. {{However, it remains unclear whether accelerated SVRG-ADMM algorithm (ASVRG-ADMM), which extends SVRG-ADMM by incorporating momentum techniques, exhibits a comparable acceleration characteristic or convergence rate in the nonconvex setting. 
To fill this gap, we consider a general nonconvex nonsmooth optimization problem and study the convergence of ASVRG-ADMM.  
By utilizing a well-defined potential energy function, we establish its sublinear convergence rate $O(1/T)$, where $T$ denotes the iteration number. }}
Furthermore, under the additional {Kurdyka-Lojasiewicz (KL)} property {which is less stringent than the frequently used conditions for showcasing linear convergence rates, such as strong convexity}, we show that the ASVRG-ADMM sequence almost surely has a finite length and converges to a stationary solution with a linear convergence rate. Several experiments on solving the graph-guided fused lasso problem and regularized logistic regression problem validate that the proposed ASVRG-ADMM performs   {better} than the state-of-the-art methods.
\\\
\\\\
\end{abstract}


\ifconfver \else \IEEEpeerreviewmaketitle} \fi

\vspace{-0.5cm}
\section{Introduction}

\label{sec:intro}
\setlength{\parindent}{2em} In recent years, machine learning has been studied and applied  extensively  in  system identification {\cite{pillonetto2014kernel}}, and automatic control {\cite{li2019performance}}.
In this paper, we consider a class of the nonconvex nonsmooth finite-sum optimization problems with a linear constraint in machine learning, as follows:
\begin{equation}\label{equ1}
\min_{\vx\in\mathbb{R}^{d_1} ,\vy\in\mathbb{R}^{d_{2}}} f(\vx)+g(\vy),\;\textup{s.t.}\;A\vx+B\vy=\vc,
\end{equation}
where $f(\vx)$ consists of the summation of $n$ components, i.e., $ f(\vx) := \frac{1}{n}\sum_{i=1}^n f_i(\vx) $ and $f_i: \mathbb{R}^{d_1}\rightarrow \mathbb{R}$ is assumed smooth, but can be nonconvex; $g(\vy)$ is convex but possibly nonsmooth; the constraint is used for encoding the structure pattern of model
parameters, where $A\in R^{d\times d_1}$, $B\in R^{d\times d_{2}}$, $\vc \in R^d$. To motivate our work, we first briefly discuss one common setting related to problem \eqref{equ1}.

Let us consider a binary classification task. Specifically, given a set of training samples $(a_i, b_i)$, $i=1,\ldots,n$, where $a_i$ is the input data and its corresponding label is $b_i\in \{1,-1\}$. Then we use the famous graph-guided fused Lasso model \cite{kim2009multivariate} to learn the parameter, which can be represented as
\begin{equation}\label{equ fused Lasso}
\min_{\vx}\frac{1}{n}\sum^{n}_{i=1} f_{i}(\vx)+\lambda_1 \|A\vx\|_{1},
\end{equation}
where the sigmoid loss function $f_{i}$ is as defined by $f_{i}(\vx)= \frac{1}{1+\exp \left(b_{i} a_{i}^{T} \vx\right)}$, which is a nonconvex function. $A$ is a given matrix decoded the sparsity pattern of the graph, which is obtained by sparse inverse covariance matrix estimation {\cite{FHT08}}. In order to solve \eqref{equ fused Lasso}, we can introduce an additional primal variable $\vy$, and
reformulate the problem (\ref{equ fused Lasso}) as a special case of (\ref{equ1})
\begin{align}
\min_{\vx,\vy} f(\vx) + g(\vy) \quad \mbox{s.t.} \ A\vx-\vy=0, \nonumber
\end{align}
where $f(\vx)=\frac{1}{n}\sum_{i=1}^n f_i(\vx)$ and $g(\vy)=\lambda_1\|\vy\|_1$.

The problem \eqref{equ1} also attracts attention and arises in many other {fields such as  statistical learning \cite{BJLZh18}}, computer vision \cite{papanikolopoulos1993adaptive} {and 3D CT image reconstruction \cite{BLZhang21}}.
In the following text, we focus on the composite finite-sum equality-constrained optimization problem  \eqref{equ1} and check the efficiency of this paper from {both algorithmic} and theoretical perspectives.

\subsection{Related work}
One of the popular methods to solve the general optimization \eqref{equ1} is ADMM algorithm {\cite{BHZ22,YJXG22,chao2021convergence}}. Specifically, at the $t$-th iteration, the typical sequential update steps are
{
\begin{equation}\label{batch ADMM}
\left\{\begin{array}{l}
\vy_{t+1}=\arg \min\limits_{\vy} \mathcal{L}_{\rho}\left(\vx_{t}, \vy, \vlambda_{t}\right), \\
\vx_{t+1}=\arg \min\limits_{\vx} \mathcal{L}_{\rho}\left(\vx, \vy_{t+1}, \vlambda_{t}\right), \\
\vlambda_{t+1}=\vlambda_{t}-\rho\left(A \vx_{t+1}+B \vy_{t+1}-\vc\right),
\end{array}\right.
\end{equation}
where $\vlambda$ is the dual variable and}
{	
\begin{align}\label{equa2}
\mathcal {L}_{\rho}(\vx,\vy,\vlambda) = &f(\vx) + g(\vy) - \langle\vlambda, A\vx+B\vy-\vc\rangle + \frac{\rho}{2} \|A\vx+B\vy-\vc\|^2
\end{align}
is an augmented Lagrangian (AL) function and ${ \rho>0}$ is a penalty parameter.} 
Many ADMM-type algorithms, such as proximal ADMM {\cite{GHWt18,YJXG22,wu2019general}}, inexact ADMM \cite{BHZ22,liu2020accelerated}, linearized/relaxed ADMM  {\cite{BHSZSZ22,Tao20}, and { consensus ADMM \cite{hong2016convergence,wang2021distributed}} have been proposed to solve the subproblem in \eqref{batch ADMM} efficiently. Recently,
a proximal AL method with an exponential averaging scheme was proposed by \cite{zhang2020global} to
handle nonconvex optimizations.

Large-scale optimization problems \eqref{equ1} typically involve a large sum of $n$ component functions, making it infeasible for deterministic ADMMs to compute the full gradient on all training samples at each iteration.
Stochastic gradient descent (SGD) has a much lower per-iteration complexity than deterministic methods by utilizing only one sample's gradient at each iteration.
Stochastic ADMM (SADMM), proposed in \cite{ouyang2013stochastic}, combines SGD and ADMM, and the update step of variable $\vx_{t+1}$ in SADMM is approximated as follows:
{
\begin{equation}\label{sADMM x}
\begin{split}
\vx_{t+1}=\arg\min_{\vx}  & ~\vx^{T}\nabla f_{i_{t}} (\vx_{t})+\frac{1}{2\eta_{t}} \|\vx-\vx_{t}\|^{2}_{Q}  +\frac{\rho}{2}\|A\vx + B\vy_{t+1} - \vc  - \frac{ \vlambda_{t}}{\rho}\|^{2},
\end{split}
\end{equation}
where we selet the random variable $i_{t}$ uniformly at random from $[n]:= \{1,\ldots,n\}$, $\eta_{t} \propto 1/\sqrt{t}$ is the step-size, and $\|\vx\|_{Q}^{2}=\vx^{T} Q \vx$ with given positive semi-definite {matrix $Q$.}
}

However, when the objective function is convex, stochastic ADMMs have been proved to own the worst-case $O(1/\sqrt{T})$ convergence rate, which is lower than the deterministic ADMMs with $O(1/T)$ convergence rate.
The gap in convergence caused by the high variances of stochastic gradients can be tackled by combining stochastic variance reduction gradient (SVRG) methods \cite{suzuki2014} with ADMM, resulting in the SVRG-ADMM in \cite{huang2016stochastic}. Unlike SADMM, variance-reduced ADMM aims to find an unbiased gradient estimator whose variance will vanish as the algorithm converges.
This encourages the selection of a larger step-size, such as a constant step-size $\eta$, to improve the convergence rate.
Specifically, we compute the gradient $\hat{\nabla} f(\vx_{t}^{s+1})$ for each iteration as follows:
{	
\begin{align}\label{f-grad-update}
\setlength{\abovedisplayskip}{-2pt}
\setlength{\belowdisplayskip}{1pt}
\hat{\nabla} f(\vx_{t}^{s+1}) = \nabla f_{i_t}(\vx_{t}^{s+1})-\nabla f_{i_t}(\widetilde{\vx}^s)+\nabla f(\widetilde{\vx}^s),  
\end{align}
{{where $i_t$ is defined in the equation (\ref{sADMM x}), and $\hat{\nabla} f(\vx_{t}^{s+1})$  is an unbiased estimator of the gradient $\nabla  f(\vx_{t}^{s+1})$, i.e., $\mathbb{E}[\hat{\nabla} f(\vx_{t}^{s+1})] =f(\vx_{t}^{s+1})$.}}
}

The Nesterov acceleration technique \cite{Nesterov2004} enjoys a fast convergence rate and has been successfully applied to train many large-scale machine learning models, including deep network models. Directly extending the Nesterov acceleration technique to the stochastic settings may aggravate the inaccuracy of the stochastic gradient and harm the convergence performance. \cite{liu2020accelerated} proposed ASVRG-ADMM, an accelerated SVRG-ADMM method by introducing a new Katyusha momentum acceleration trick \cite{AllenZhu2017KatyushaTF} into SVRG-ADMM. {{ Table \ref{tab algs} presents a comparison of several stochastic ADMM algorithms. All algorithms in the table incorporate the VR technique, while only the last three employ momentum acceleration techniques.
}}

Although ASVRG-ADMM has been successful in convex problems, its behavior in nonconvex nonsmooth problems remains largely unknown.
Many important applications, such as signal/image processing \cite{GezWu22}, computer vision \cite{papanikolopoulos1993adaptive}, and video surveillance \cite{YPongC17}, involve nonconvex objective functions that are beyond the scope of the theoretical conditions under which ASVRG-ADMM \cite{liu2020accelerated} has been proven to converge. 
{Therefore, there exists a gap in the theoretical convergence analysis between extant ASVRG-ADMM in the convex scenario and ASVRG-ADMM in the nonconvex scenario.}

{ To fill this gap, this paper addresses two crucial problems}:
\begin{description}
\item[Q1.] Whether the ASVRG-ADMM algorithm can converge to the stationary point for the common nonconvex nonsmooth finite-sum optimization?
\item[Q2.] If the algorithm converges, can we establish  its linear convergence rate such as R-linear\footnote{

{\setlength{\abovedisplayskip}{1pt}
\setlength{\belowdisplayskip}{1pt}
The sequence $\left\{x_n\right\}_{n \in \mathbb{N}}$   converges to $x^{*}$ R-linearly if
$$\left\|x_n-x^{*}\right\| \leq  y_n, $$
for all $n$, and the sequence $\left\{y_n\right\}_{n \in \mathbb{N}}$ converges to zero Q-linearly.
Then one can always find a suitable sequence}
{\setlength{\abovedisplayskip}{1pt}
\setlength{\belowdisplayskip}{1pt} $\left\{y_n\right\}_{n \in \mathbb{N}}$ as $y_n = c*\xi^n$, such that
$$\left\|x_n-x^{*}\right\| \leq c*\xi^n,$$
for all $n$, where $0<\xi<1$.}
} {convergence?}
\end{description}


\begin{table*}[t]
\footnotesize
\centering
\caption{Comparison of convergence rates of some stochastic ADMM algorithms.}
\label{tab algs}

\begin{tabular}{|c|c|c|c|c|c|}
\hline
\textbf{Algorithms}  & $\textbf{Strongly convex}$ & $\textbf{Convex}$   & $\textbf{Nonconvex}$ & $\textbf{Momentum}$  & $\textbf{KL property}$\\  \cline{1-6}
SAG-ADMM \cite{huang2016stochastic}  & $\diagdown$\ & $\diagdown$\  &  $\mathcal{O}(1/T) $  &\XSolid & \XSolid  \\ \cline{1-6}
SDCA-ADMM \cite{suzuki2014}   & Linear Rate & $\diagdown$\   & $\diagdown$\ & \XSolid & \XSolid \\ \cline{1-6}
SVRG-ADMM \cite{huang2016stochastic}    & $\diagdown$\ & $\diagdown$\ &  $ \mathcal{O}(1/T) $ & \XSolid & \XSolid  \\ 
\cline{1-6}
SPIDER-ADMM \cite{huang2019faster}    & $\diagdown$\  & $\diagdown$\  &  $ \mathcal{O}(1/T) $ & \XSolid & \XSolid  \\
\cline{1-6}
ASVRG-ADMM \cite{liu2020accelerated}    & Linear Rate & $ \mathcal{O}(1/T^{2}) $ &   $\diagdown$\ & \Checkmark & \XSolid \\ 
\cline{1-6}
AS-ADMM \cite{BHZ22}    &  Linear Rate & $ \mathcal{O}(1/T) $ &   $\diagdown$\ & \Checkmark & \XSolid \\ \cline{1-6}
\textbf{ASVRG-ADMM (ours)}    &$\diagdown$\ & $\diagdown$\ &  $\!\mathcal{O}(1/T)\!$ { \& linear rate}   & \Checkmark &\Checkmark \\ \cline{1-6}

\end{tabular}\label{b-table1}
\end{table*}

\subsection{Contributions}
In this paper, we develop  ASVRG-ADMM for the nonconvex and nonsmooth problems (\ref{equ1}).  ASVRG-ADMM is inspired by the momentum \cite{liu2020accelerated,BHZ22,AllenZhu2017KatyushaTF} and the variance reduction \cite{suzuki2014,huang2016stochastic} technique. In addition, it has several new features. First, compared with \cite{liu2020accelerated}, we consider the nonconvex objective function rather than the convex function, which poses a major obstacle to establishing the convergence of our algorithm without
the sufficiently decreasing property of the AL function.
Second, compared with SVRG-ADMM in \cite{huang2016stochastic}, 
{{our proposed algorithm employs the Katyusha momentum technique to improve  the accuracy of the stochastic gradient and facilitate faster convergence in practice.}}
In summary, our main contributions include four folds as follows:

\begin{itemize}
\item[(C1)] We propose a novel accelerated stochastic ADMM (ASVRG-ADMM) for the nonconvex and nonsmooth problem (\ref{equ1}), which enjoys the advantages of both the momentum acceleration technique \cite{AllenZhu2017KatyushaTF,liu2020accelerated} and  the variance reduction {technique}  \cite{huang2016stochastic,suzuki2014}.

\item[(C2)] {{By introducing a potential energy function related to the AL function of (\ref{equ1}), we establish the sublinear convergence rate $O(1/T)$, then we answer {\textbf{Q1}}.
While ASVRG-ADMM has the same convergence rate as SVRG-ADMM \cite{huang2016stochastic}, theoretically, we can achieve faster decay of the potential energy function than SVRG-ADMM by selecting the optimal momentum parameter. Furthermore, our algorithm achieves better convergence in numerical experiments.

}}

\item[(C3)]
We establish the almost sure R-linear convergence rate of ASVRG-ADMM by utilizing the extra KL property to answer {\textbf{Q2}}, as  {{listed in Table \ref{tab algs}
and demonstrated in  Theorem \ref{thm 3}}}.
Answering {\textbf{Q2}} under the nonconvex and nonsmooth setting is not a trivial task due to the inclusion of both the momentum term and variance reduction technique.
To the best of our knowledge, this is the first almost sure linear convergence result for the accelerated stochastic ADMM algorithm with variance reduction technique.

\end{itemize}

Various numerical experiments demonstrate the effectiveness of our ASVRG-ADMM, as it exhibits both lower variance and faster convergence compared to several advanced stochastic ADMM-type methods \cite{huang2016stochastic,ouyang2013stochastic}.

{\textbf{ Synopsis}:} The remainder of this paper is organized as follows. Section \ref{sec2} introduces the notations and mild assumptions. {We present the framework of  ASVRG-ADMM in
Section \ref{sec 3}. We
study its convergence 
and  extends} the theoretical analysis of the almost sure linear convergence rate under the extra {KL} property in Section \ref{sec: conv analysis}. Section \ref{sec 5} presents some numerical results to demonstrate the effectiveness and efficiency of the method. Finally, we conclude the paper with discussions.

{\bf Notation:} { The symbol $\| \cdot \|$ denotes the Euclidean norm of
a vector (or the spectral norm of a matrix)}, and $\|\cdot\|_{1}$ is
the $\ell_{1}$ -norm, i.e., $\|\vx\|_{1}=\sum_{i}\|x_{i}\|$. { We use
$\mathbb{R}, \mathbb{R}^{d}$, and $\mathbb{R}^{d \times d_{1}}$ to denote} the sets of real numbers, $d$ dimensional real column vectors, and $d \times d_{1}$ real matrices, respectively. 
Let $(\Omega, \mathcal{F}, P)$ be the probability space, $\vx: \Omega \rightarrow \mathbb{R}^d$ be a random variable, $\mathbb{E}[\vx]$ denote the expectation,
$\mathbb{E}[\vx \mid \mathcal{F}^{s}_{t}]$ denote the conditional expectation of $\vx$ given sub-$\sigma$-algebra $\mathcal{F}^{s}_{t}$, and employ the abbreviations ``a.s." for ``almost surely".		
Let $I$ denote the identity matrix and $\mathbf{0}$ denote the zero matrix.
We denote by $\nabla f(\vx)$ the gradient of $f(\vx)$ if it is differentiable, or $\partial f$ any of the  subgradients of the function $f(.)$.
{Let $\sigma^{A}_{min}$ and $\sigma^{A}_{max}$ be} the smallest and largest eigenvalues of matrix $AA^T$, respectively;
and let $\phi_{\min}$ and $\phi_{\max}$ denote the smallest and largest eigenvalues of positive matrix $Q$, respectively.
For a nonempty closed set $\mathcal{C}, dist(\vx, \mathcal{C})=\inf _{\vy \in \mathcal{C}}\|\vx-\vy\|$ denotes the distance from $\vx$ to set $\mathcal{C}$.

\section{Preliminaries}\label{sec2}
\subsection{Basic Assumptions and Definitions}

Before establishing the convergence analysis, we clarify the following proper assumptions. 
\begin{assumption}\label{assum grad f}
For $\forall i \in \{1,2,\cdots,n\}$, each function $f_i$ possesses a
Lipschitz continuous gradient with the constant $L_i>0$, that is
{
\begin{align}
\|\nabla f_i(\vx_1)-\nabla f_i(\vx_2)\| \leq L \|\vx_1 - \vx_2\|, \ \forall \vx_1,\vx_2 \in R^{d_1},\nonumber
\end{align}}
where $L=\max_i L_i$.
\end{assumption}

\begin{assumption}\label{assum2}
$f(\vx) $ and $g(\vy)$ are  lower bounded.
\end{assumption}

\begin{assumption}[feasibility]\label{assum4}
{ ${(\operatorname{Im}(B)\cup \vc)}\subseteq\operatorname{Im}(A)$, where $\operatorname{Im}(\cdot)$ returns the image of a matrix.
}
\end{assumption}

\begin{assumption}[Lipschitz sub-minimization paths]\label{Lip sub path}
For any fixed $\mathbf{x}, \operatorname{argmin}_y\{f(\mathbf{x}) + g(\vy): B \vy=\vu\}$ has a unique minimizer. In addition, $H: \operatorname{Im}(B) \rightarrow \mathbb{R}^{d_2}$ defined by $H(\vu) \triangleq \operatorname{argmin}_\vy \{f(\mathbf{x}) + g(\vy): B \vy=\vu\}$ is a Lipschitz continuous map.
\end{assumption}

Assumption \ref{assum grad f} is a standard smoothness assumption. Note that Assumptions  \ref{assum grad f}-\ref{assum2} are satisfied {for} the case of $f(\vx)=\frac{1}{1+\exp \left(\vx\right)}$, { $g(\vy)=\lambda_1\|\vy\|_1$}. Assumptions \ref{assum4}-\ref{Lip sub path} are proposed in \cite{wang2019global}, which weaken the full column rank assumption typically imposed on matrices $A$ and $B$. 
{{Moreover, Assumptions \ref{assum4}-\ref{Lip sub path} are the mild assumptions commonly used to  ensure that the dual variable $\vlambda$ is controlled by the primal variable,  see  Lemma \ref{lemma of lambda}.
}}

Next, we introduce the critical point and KL property used in this paper.


\begin{definition}\label{kkt}
The point $\left(\mathcal{X}^{*}, \mathcal{Y}^{*}\right)$ is denoted as the feasible solution set of the problem (\ref{equ1}). 
When there exist a optimal point $\left(\vx^{*},  \vy^{*}\right)$ and a dual variable $\vlambda^{*}$ for the problem (\ref{equa2}), the KKT conditions are satisfied as follows:
{
\begin{equation}
\left\{\begin{array}{l}
A^{T} \vlambda^{*}=\nabla f\left(\vx^{*}\right),  \\
B^{T} \vlambda^{*} \in \partial g\left(\vy^{*}\right),  \\
A \vx^{*}+B \vy^{*}-\vc=0.
\end{array}\right. \nonumber
\end{equation}
We also say that $\left(\vx^{*}, \vy^{*}, \vlambda^{*}\right)$ {satisfy} the KKT conditions is exactly a critical point of the AL function, denoted as $ \emph{crit}~ \mathcal{L}_{\rho}$.}
\end{definition}


\begin{definition}[{KL} property  \cite{attouch2010proximal}]\label{KŁ1}

A proper lower semicontinuous function $F: \mathbb{R}^{n} \to \mathbb{R} \cup \{+\infty\}$ is said to possess the KL property at $\vx^{*} \in$ dom $\partial F$ if there exist a constant $\eta \in (0, \infty]$, a neighborhood $U$ of $\vx^{*}$ and a continuous concave function $\varphi: [0, \eta) \to \mathbb{R}_{+}$, satisfying the following requirements:
{	\setlength{\abovedisplayskip}{2pt}
\setlength{\belowdisplayskip}{2pt}
\begin{itemize}
\item   $\varphi(0) = 0,~\varphi \in C^{1}((0, \eta))$, and  $\nabla \varphi(s) > 0$ for all $s \in (0, \eta)$;
\item for all $\vx \in U \cap [F(\vx^{*}) < F < F(\vx^{*})+ \eta]$, the KL inequality holds:
{
\begin{equation}\label{KŁ inequ}
\nabla \varphi(F(\vx) - F(\vx^{*}) ) dist (0, \partial F(\vx)) \geq 1.
\end{equation}}
\end{itemize}}
\end{definition}

{{Many effective functions satisfy the KL property, including but not limited to $\|\mathbf{x}\|_1$,  $\|\mathbf{A} \mathbf{x}-\mathbf{b}\|^2$, log-exp, and the logistic loss function $\psi(t)=\log \left(1+e^{-t}\right)$. For more examples, readers can refer to \cite{attouch2010proximal,bolte2014proximal}, \cite{milzarek2023convergence,chouzenoux2023kurdyka}, and \cite{yashtini2022convergence}(see Page 919).
Therefore, the potential energy function defined in (\ref{def of psi}) can easily satisfy the KL property.}}

\begin{remark}
When function $F$ in Definition \ref{KŁ1} is differentiable, then the {KL} property becomes the PL property {\cite{yi2022zeroth}}, which is {much weaker than the strongly convex property}, and widely used to show the global convergence in the field of deep learning. However, the {PL} property only deals with smooth differentiable functions. In this paper, the considered problem \eqref{equ1} not only has the smooth term but also has the nonsmooth term, thus we use the {KL} property to show the linear convergence.
\end{remark}

\vspace{-0.0cm}

\section{ Developments of ASVRG-ADMM}\label{sec 3}
\vspace{-0.0cm}
We propose ASVRG-ADMM, which incorporates momentum and variance reduction techniques to ensure fast convergence of the nonsmooth and nonconvex problem (\ref{equ1}). The algorithm, presented in Algorithm \ref{alg1}, consists of $S$ epochs, each containing $m$ iterations with $m$ typically chosen to be $O(n)$.

\begin{algorithm}[t]
\caption{ASVRG-ADMM {for solving (\ref{equ1})}}
\KwIn{parameter $m, T, S=[T/m], \eta,  \gamma, \rho, 0 \leq \theta \leq 1,$ and initial values of $\widetilde{\vx}^0=\vx_m^0$, $\vy_m^0$, $\vz_m^0 = \vx_m^0$ and $\vlambda_m^0$. }
\KwOut{Iterate $\vx$ and $\vy$ chosen  { from} $\{(\vx_{m}^S,\vy_{m}^S)\}$.}
\For{$s=0,1,\cdots,S-1$}{
$\vx_0^{s+1}=\vx_{m}^s$, $\vy_0^{s+1}=\vy_{m}^s$ and $\vlambda_0^{s+1}=\vlambda_{m}^s$;\\
$\nabla f(\widetilde{\vx}^s)=\frac{1}{n}\sum_{i=1}^n\nabla f_i(\widetilde{\vx}^s)$;\\
\For{$t=0,1,\cdots,m-1$}{
{Uniformly and randomly  choose a random variable} $i_t$ from $\{1,\cdots,n\}$ {{,  and compute $\hat{\nabla} f(\vx_{t}^{s+1})$ in (\ref{f-grad-update}) }};\\
\begin{align}
& \label{y-update} \vy^{s+1}_{t+1}=\arg\min_\vy \mathcal {L}_{\rho}(\vx^{s+1}_t,\vy,\vlambda_t^{s+1}); \\
&  \vz^{s+1}_{t+1}   =   \vz^{s+1}_{t}\ -\frac{\eta}{\gamma\theta} \Big[\hat{\nabla} f(\vx_{t}^{s+1})  +\rho A^{T}  \Big(A\vz^{s+1}_{t}  +B\vy^{s+1}_{t+1}-\vc-\frac{\vlambda^{s+1}_{t}}{\rho}\Big)\Big]; \label{z-update}  \\
& \vx^{s+1}_{t+1}= \theta \vz^{s+1}_{t+1} + (1-\theta)\widetilde{\vx}^s;  \label{x-update} \\
& \vlambda_{t+1}^{s+1} = \vlambda_{t}^{s+1}-\rho(A \vz_{t+1}^{s+1}+B\vy_{t+1}^{s+1}-\vc); \label{dual-update}
\end{align}
} 
$\widetilde{\vx}^{s+1}= \vx_m^{s+1}$.
}

\label{alg1}
\end{algorithm}


\subsection{Update of Primal Variable $\vy$}
The first step is to update the primal  variable $\vy$ in (\ref{y-update}) by solving
{
\begin{align}\label{update-y1}
\vy^{s+1}_{t+1} &=\arg\min_\vy \mathcal {L}_{\rho}(\vx^{s+1}_t,\vy,\vlambda_t^{s+1}) \\
\nonumber     & = \arg\min_\vy g(\vy)+ \frac{\rho}{2} \|A\vx^{s+1}_t+B\vy-\vc-\frac{\vlambda_t^{s+1}}{\rho} \|^2.
\end{align}}

If $g(\vy)=\lambda_1 \|\vy\|_1$, one can observe that the optimization \eqref{update-y1} has closed form solution with
{	
\begin{align*}
\vy^{s+1}_{t+1} =&\max\left(\mathbf{0}, A\vx^{s+1} - \frac{\vlambda_t^{s+1}}{\rho}-\frac{\lambda_1}{\rho}\right) - \max\left(\mathbf{0}, -A\vx^{s+1} + \frac{\vlambda_t^{s+1}}{\rho}-\frac{\lambda_1}{\rho}\right),
\end{align*}
where ${\lambda_1}$ is the regularization parameter.}



\subsection{Update  of Auxiliary Variable $\vz$}


Inspired by Katyusha \cite{AllenZhu2017KatyushaTF,liu2020accelerated} acceleration tricks, the second step is adopting an auxiliary variable $\vz$  before updating the primal variable $\vx$. 
A surrogate function is first introduced to approximate the objective function in (\ref{batch ADMM}) and the update rule of $\vz$ is formulated as follows: 
{		\begin{align}\label{update-z}
\vz^{s+1}_{t+1}=\arg\min_{\vz}~& (\vz- \vz^{s+1}_{t})^{T}\hat{\nabla} f(\vx_{t}^{s+1}) + \frac{\theta}{2\eta} \|\vz- \vz^{s+1}_{t}\|^{2}_{Q} +\frac{\rho}{2}\|A \vz+B\vy^{s+1}_{t+1}-\vc - \frac{\vlambda^{s+1}_{t}}{\rho}\|^{2},
\end{align}
where  $\eta>0$ is the learning rate or step-size, and $\hat{\nabla} f(\vx_{t}^{s+1})$ is the stochastic variance reduced gradient
estimator defined in \eqref{f-grad-update}.}

%
%
%
%
%

In general, the direct solution of (\ref{update-z}) can be challenging or even infeasible since it requires the inversion of $\frac{\theta}{\eta}Q + \rho A^TA$.
To alleviate this computational burden, the inexact Uzawa method can be employed to linearize the last term (\ref{update-z}). 
Specifically, we can select $Q=\gamma I -\frac{\eta\rho}{\theta}A^{T}A$ with $\gamma \geq 1+\frac{\eta \rho \|A^{T} A\|_{2}}{\theta}$ to ensure that $Q\succ I$, which results in the simpler closed-form solution of (\ref{update-z}) as  (\ref{z-update}).

\subsection{Momentum Accelerated Update Rule for  $\vx$}
The next step is our momentum accelerated update rule for the primal variable $\vx$, as given below:
{	
\begin{equation}\label{update-x}
\vx^{s+1}_{t+1}= \theta \vz^{s+1}_{t+1} + (1-\theta)\widetilde{\vx}^s = \widetilde{\vx}^s + \theta(\vz^{s+1}_{t+1} - \widetilde{\vx}^{s}),
\end{equation}
where $\theta$ is a momentum parameter, and $\theta(\vz^{s+1}_{t+1} - \widetilde{\vx}^{s})$ is a momentum term,}
{which uses the final output iterate from the previous epoch (updated every m iterations in the inner loop)}, i.e., $\widetilde{\vx}^{s}$, to speed up our algorithm. {{
In (\ref{update-x}), we leverage the Katyusha momentum technique \cite{AllenZhu2017KatyushaTF}, which uses a convex combination of the latest $\vz^{s+1}_{t+1}$ and snapshot $\widetilde{\vx}^{s}$,
whereas the Nesterov-type momentum technique uses a nonconvex extrapolation of the two latest iterates.
Compared with Nesterov momentum, Katyusha momentum ensures that
the iterates do not deviate too far from $\widetilde{\vx}^{s}$, which ensures accurate gradient estimation. }}

\subsection{Update  of Dual Variable $\vlambda$}
The final step is to update the dual variable by
$$\vlambda_{t+1}^{s+1} = \vlambda_{t}^{s+1}-\rho(A \vz_{t+1}^{s+1}+B\vy_{t+1}^{s+1}-c).$$
Note that we use auxiliary variable $\vz^{s+1}_{t+1}$ rather than the primal variable $\vx^{s+1}_{t+1}$ \cite{ouyang2013stochastic} in the dual update scheme, which helps us to establish the convergence of the proposed algorithm under nonconvex setting.

Before ending the section, let's make a few remarks about ASVRG-ADMM.

\begin{remark}
{{ Compared with SVRG-ADMM \cite{huang2016stochastic}, our ASVRG-ADMM incorporates a new acceleration step \eqref{update-x} and its sublinear and linear convergence have been established under weaker Assumption \ref{assum4}- \ref{Lip sub path}. Additionally, if we set $\theta$ in the update step of $\vx^{s+1}_{t+1}$ to 1, it results in the algorithm degenerating into SVRG-ADMM \cite{huang2016stochastic}, where $\vx^{s+1}_{t+1} = \vz^{s+1}_{t+1}$, which can be evidenced by the overlapping descent curves of ASVRG-ADMM1 ($\theta = 1$) and SVRG-ADMM in Figure \ref{fig_theta and rho} (left).}}
\end{remark}

\section{ Convergence  {Analysis}}\label{sec: conv analysis}
This section establishes the sublinear and linear convergence of our proposed ASVRG-ADMM for nonconvex and nonsmooth problems.
In the literature, the AL function \cite{wang2019global} is frequently used as the potential energy function with sufficiently decreasing property (see \cite{wang2019global}). However, in the nonconvex setting, our Algorithm \ref{alg1} deviates from this approach.
To ensure sufficient decrease and convergence, we present the following lemmas and theorems with a practical sequence of potential energy functions $\{(\Psi^{s}_{t})_{t=1}^m\}_{s=1}^S$ defined as follows:
{	
\begin{align}\label{def of psi}
\setlength{\abovedisplayskip}{2pt}
\setlength{\belowdisplayskip}{2pt}
\Psi^{s}_{t}:= & 
\mathcal {L}_{\rho}(\vx^{s}_{t},\vy^{s}_{t},\vlambda^{s}_{t})+ \beta_5 \|\vx^{s}_{t}-\vx^{s}_{t-1}\|^2+ h^{s}_{t}(\|\vx^{s}_{t}-\widetilde{\vx}^{s-1}\|^2+ \|\vx^{s}_{t-1}-\widetilde{\vx}^{s-1}\|^2),
\end{align}
where the positive sequence $\{(h^s_t)_{t=1}^m \}_{s=1}^S$ and the parameter $\beta_5>0$ will be decided in the following Lemma \ref{detail lemma of h}.

Let $\{(\vx^{s}_{t})_{t=1}^m\}_{s=1}^S$ be a stochastic process adapted to the filtration $\{(\mathcal{F}^{s}_{t})_{t=1}^m\}_{s=1}^S$,  where $\mathcal{F}^{s}_{t}=\sigma\left(\mathbf{x}_1^1, \mathbf{x}_2^1, \ldots \mathbf{x}^{s}_{t}\right)$, $\sigma(\cdot)$ is the $\sigma$-field, and we abbreviate
$\mathbb{E}[\cdot \mid \mathcal{F}^{s}_{t}]$ as $\mathbb{E}^{s}_{t}[\cdot ]$, see more details in \cite{BLZhang21,milzarek2023convergence}. Additionally, we define $\Psi^{*}$ as $\min_{s,t} \Psi_t^s$.}
{{ 
Now, we first provide an outline of the proof in this section. 
\begin{itemize}
\item Firstly, we propose a new potential energy function, which differs from the one in \cite{huang2016stochastic}, and prove the monotonicity of this function.
\item Next, we leverage the monotonicity property of the potential energy function to establish the convergence rate of the newly defined sequence in (\ref{equ 21}).
\item Lastly, we leverage the KL property along with finite-length property in Lemma \ref{thm 2} to achieve a almost sure  linear convergence result for ASVRG-ADMM.
\end{itemize}
}}

Our convergence analysis will utilize a bound on the term $ \mathbb{E}^{s+1}_{t}\|\vlambda^{s+1}_{t+1}-\vlambda^{s+1}_{t}\|^2$, which is given by Lemma \ref{lemma of lambda}.

\begin{lemma}[Upper bound of $ \mathbb{E}^{s+1}_{t}\|\vlambda^{s+1}_{t+1}-\vlambda^{s+1}_{t}\|^2$]\label{lemma of lambda}
Given Assumptions \ref{assum grad f}-\ref{assum4}, for $\{(\vx^{s}_t,\vy^{s}_t,\vlambda^{s}_t)_{t=1}^m\}_{s=1}^S$   generated by the Algorithm \ref{alg1},
{ we can have the following result}
{	
\begin{align}\label{upp1}
\mathbb{E}^{s+1}_{t}[\|\vlambda^{s+1}_{t+1}-\vlambda^{s+1}_{t}\|^2] 
\leq  &\frac{5L^2}{\sigma^{A}_{min}}  [\|\vx^{s+1}_{t}-\widetilde{\vx}^{s}\|^2] +   \frac{5L^2}{\sigma^{A}_{min}}\|\vx^{s+1}_{t-1}-\widetilde{\vx}^{s}\|^2
+ \frac{5\phi^2_{\max}}{\sigma^{A}_{min} \eta^2} \mathbb{E}^{s+1}_{t}[\|\vx^{s+1}_{t+1}-\vx^{s+1}_t\|^2] \nonumber\\
&+ \frac{5(L^2 \eta^2+\phi^2_{\max})}{\sigma^{A}_{min}\eta^2}\|\vx^{s+1}_{t}-\vx^{s+1}_{t-1}\|^2.
\end{align}}
\end{lemma}

\begin{proof}
See Appendix \ref{app:lemma of h}.
\end{proof}

\begin{lemma}\label{detail lemma of h}
For the sequence  $\{(\vx^{s}_t,\vy^{s}_t,\vlambda^{s}_t)_{t=1}^m\}_{s=1}^S$ generated by Algorithm \ref{alg1}, assuming that Assumptions \ref{assum grad f}-\ref{assum4} hold, and the positive sequences  $\{(h_t^s)_{t=1}^m\}_{s=1}^S$ and $\{(\Gamma^{s}_t)_{t=1}^m\}_{s=1}^S$ exist, we choose positive parameters $\eta$, $\theta$, $l_{1}$, $l_{2}$, $l_{3}$, $\alpha_1$, and $\rho$ such that
{	
$$   \Gamma^{s}_t >0 , \ \forall t \in [m], \ \forall s \in [S],$$ 
then the sequence $\{(\Psi^{s}_{t})_{t=1}^m\}_{s=1}^S$ is sufficiently decreasing:}
{	
\begin{align}\label{equ of decrease}
& \mathbb{E}_{t}^{s+1}[\Psi^{s+1}_{t+1}] \leq \Psi^{s+1}_{t}-\Gamma^{s+1}_t \mathbb{E}_{t}^{s+1}[\|\vx^{s+1}_{t+1}-\vx^{s+1}_t\|^2] - \left[h^{s+1}_{t+1}+\left(1+\alpha_1 \right)(h^{s+1}_{t+1}+ \beta_1)\right] \|\vx^{s+1}_{t-1}-\widetilde{\vx}^{s}\|^2.\end{align}}

\end{lemma}

\begin{proof}
The specific expression of the tuple $(\beta_1,\beta_2,\beta_3,\beta_4,\beta_5,\beta_6)$ and the sequences $\{(h_t^s)_{t=1}^m\}_{s=1}^S$, $\{(\Gamma^{s}_t)_{t=1}^m\}_{s=1}^S$ are provided in Appendix \ref{app:lemma 2}.
\end{proof}

\begin{remark}\label{remark 4}
As we see, there are some crucial parameters in Lemma \ref{detail lemma of h}. In this remark, we further clarify how to {choose these parameters $\theta,\eta,\rho$.} One can observe that there exists a parameter $\Gamma^s_t$ in the right side of the inequality  \eqref{equ of decrease}. If $\Gamma^s_t$ promotes larger, then the sequence $\{(\Psi^{s}_{t})_{t=1}^m\}_{s=1}^S$ decreases faster.
\begin{itemize}
\item Firstly, we rewritten $\Gamma^s_t$ defined in \eqref{def of gamma} as follows:
\begin{align}
\Gamma^s_{t}=&F(\theta)+H(\eta) -\frac{L}{2}  - \left(\frac{1}{\rho}+\frac{1}{2 l_{1}}\right)\frac{5L^2}{\sigma^{A}_{min}}-h_{t+1}^{s}(1+\frac{1}{\alpha_1}), \label{three part}
\end{align}
where
{
\begin{align}
F(\theta) = &\frac{\rho\sigma^{A}_{min}(l_{2}-(1-\theta))}{2\theta^2 l_{2}} - \left(\frac{\rho+l_{1}}{2}\right) \left(\frac{1-\theta}{\theta}\right)^{2}\sigma^{A}_{max}(1+\frac{1}{\alpha_1}), \label{fff} \\
H(\eta)= & \frac{\phi_{\min}}{\eta} - \frac{10\phi^2_{\max}  }{\sigma^{A}_{min}\eta^2}(\frac{1}{\rho} + \frac{1}{2l_1}),\label{hhh}
\end{align}}
\item Secondly, in order to obtain the maximum value of $\Gamma^s_t$, then the optimal parameters $\theta^*$ and $\eta^*$ can be obtained
{
\begin{align}\label{opt theta1}
\theta^* &= \frac{2 l_2 (\rho + l_1)(1+\alpha_1)\sigma^{A}_{max}- 2\alpha_1 \rho \sigma^{A}_{min}(l_2+1)}{2 l_2 (\rho + l_1)(1+\alpha_1)\sigma^{A}_{max} + \alpha_1 \rho \sigma^{A}_{min}(l_2+1)}\\
\frac{1}{\eta^*} &=  \frac{\phi_{\min } \sigma^{A}_{min} }{20 \phi_{\max }^{2}\left(\frac{1}{\rho}+\frac{1}{2 l_{1}}\right)}, \label{opt eta1}
\end{align}}
\item Finally, from \eqref{opt theta1} and \eqref{opt eta1}, we can see that the optimal $\theta^*$ and $\eta^*$ become larger, the parameter $\rho$ should be adjusted to smaller.
\end{itemize}

\begin{remark}\label{remark 5}
Now we show that there also exit a lower bound for  the dual step size $\rho$.
Taking the optimal value of $\eta^*$ in (\ref{opt eta1}) into $\Gamma^s_{t}$ in (\ref{three part}), we study the requirement of  $\rho$ with the inequality $\Gamma^s_{t}>0$, which is simplified as follows:
{	\setlength{\abovedisplayskip}{10pt}
\setlength{\belowdisplayskip}{10pt}
\begin{align}
0< &\frac{\phi_{\min }^{2} \sigma^{A}_{min} }{40 \phi_{\max }^{2}\left(\frac{1}{\rho}+\frac{1}{2 l_{1}}\right)} - \frac{L}{2}  + \frac{\rho\sigma^{A}_{min}(l_{2}-(1-\theta))}{2\theta^2 l_{2}} \nonumber\\
&- \left(\frac{1}{\rho}+\frac{1}{2 l_{1}}\right)\frac{5L^2}{\sigma^{A}_{min}}  -(h_{t+1}^{s}+\beta_1)(1+\frac{1}{\alpha_1}). \nonumber
\end{align}}
{Considering} the setting of $\sigma^{A}_{min} \rho >1$, the parameter $\rho$ should satisfy the following inequality
\begin{align}
&10 L^{2}+\left(h_{t+1}^{s}+\beta_{1}\right)\left(1+\frac{1}{\alpha_{1}}\right)\nonumber\\
< & \frac{\phi_{\min }^{2} \sigma_{\min }^{A} \theta^{2} l_{2}+40 \sigma_{\min }^{A}\left[l_{2}-(1-\theta)\right] \phi_{\max }^{2}}{80 \phi_{\max }^{2} \theta^{2} l_{2}} \rho,  \nonumber
\end{align}
which indicates
{	\setlength{\abovedisplayskip}{10pt}
\setlength{\belowdisplayskip}{10pt}
\begin{equation} \label{inequ rho}
\rho>\frac{80 \phi_{\max }^{2} \theta^{2} l_{2}\left[10 L^{2}+\left(h_{t+1}^{s}+\beta_{1}\right)\left(1+\frac{1}{\alpha_{1}}\right)\right]}{\phi_{\min }^{2} \sigma_{\min }^{A} \theta^{2} l_{2}+40 \sigma_{\min }^{A}\left[l_{2}-(1-\theta)\right] \phi_{\max }^{2}} .
\end{equation}}
\end{remark}

%
%
%
%
%
%
\end{remark}	

Based on the above important descending Lemma \ref{detail lemma of h}, we next
analyze the convergence and iteration complexity of the ASVRG-ADMM  with the aid of a simple notation:
{		
\begin{align}\label{equ 21}
R^{s}_t := &\mathbb{E}\|\vx^{s}_{t}-\widetilde{\vx}^{s-1}\|^2 + \mathbb{E}\|\vx^{s}_{t-1}-\widetilde{\vx}^{s-1}\|^2  + \mathbb{E}\|\vx^{s}_{t+1}-\vx^{s}_t\|^2 + \mathbb{E}\|\vx^{s}_{t}-\vx^{s}_{t-1}\|^2.
\end{align}}
We also denote
{	\setlength{\abovedisplayskip}{-2pt}
\setlength{\belowdisplayskip}{-2pt}
\[
(\hat{t},\hat{s}) =
\mathop{\arg\min}_{1 \leq t\leq m,\ 1 \leq s\leq S}R^{s}_t.
\]}

\begin{theorem}\label{thm 1}
Assuming that Assumptions \ref{assum grad f}-\ref{assum4} hold, let $\{(\vx^{s}_t,\vy^{s}_t,\vlambda^{s}_t)_{t=1}^m\}_{s=1}^S$ be the sequence generated by Algorithm \ref{alg1}, using the identical notations as in Lemma \ref{detail lemma of h}. For $\tau = \min\big\{\gamma, \omega \big\}>0$, where $\gamma=\min\limits_{t,s} \Gamma^s_t$ and $\omega = \frac{5L^2}{\sigma^{A}_{min} \rho}$, the following can be derived:	
{	\setlength{\abovedisplayskip}{2pt}
\setlength{\belowdisplayskip}{2pt}
\begin{align}
R^{\hat{s}}_{\hat{t}} = \min_{s,t} R^{s}_t \leq \frac{1}{\tau T} \mathbb{E}(\Psi^{1}_{1} - \Psi^*),
\end{align}
where $\Psi^*$ denotes a lower bound of the sequence $\{(\Psi^{s}_{t})_{t=1}^m\}_{s=1}^S$.}


\end{theorem}

\begin{proof}
See	 Appendix \ref{app:theorem 1}.
\end{proof}

Theorem \ref{thm 1}  shows that Algorithm  \ref{alg1} converges with the worst-case $\mathcal{O}(1/T)$ convergence rate, where $T=mS$.

\begin{lemma}[The Finite Length Property] \label{thm 2}
{Let the sequence $\left\{\vw_{t}=\left(\vx_{t}, \vy_{t}, \vlambda_{t}\right)\right\}$ be generated by  Algorithm \ref{alg1} under  Assumptions \ref{assum grad f}-\ref{Lip sub path}   and  $\Psi_{t}$ {  has the  KL} property at the stationary point $\left(\vx^{*}, \vy^{*}, \vlambda^{*}\right)$ with
{KL} exponent $\Tilde{\mu}\in [0; 1)$ (see the Definition \ref{KŁ1}). Suppose that
$f$ and $g$ are semi-algebraic functions satisfying the {KL} inequality, then $\left\{\vw_{t}\right\}$ has a finite length, that is,}
$$
\sum_{t=0}^{+\infty}\left\|\vw_{t+1}-\vw_{t}\right\|<+\infty \quad \text{a.s}.
$$
{Moreover,}  $\left\{\vw_{t}\right\}$ converges a.s. to a critical point of $\mathcal{L}_{\rho}(\cdot)$ .
\end{lemma}

\begin{proof}
See	 Appendix \ref{app:thm 2}.
\end{proof}

Now, combined with the KL property defined in Definition \ref{KŁ1} (see \cite{attouch2010proximal}), we can obtain the linear convergence rate of ASVRG-ADMM.

\begin{theorem}[Linear Convergence]\label{thm 3}
For notational simplicity, we also omit the upper script $s$. The sequence $\left\{\vw_{t}=\left(\vx_{t}, \vy_{t}, \vlambda_{t}\right)\right\}$ is  generated by Algorithm \ref{alg1}, and it converges to $\left\{\vw^{*}=\left(\vx^{*}, \vy^{*}, \vlambda^{*}\right)\right\}$. Additionally, we define $\varphi(s)=\tilde{c} s^{1-\Tilde{\mu}}, \Tilde{\mu} \in[0,1), \tilde{c}>0$.  Assume that $\Psi_{t}$ defined in (\ref{def of psi})  has the  {KL} property at the stationary point $\left(\vx^{*}, \vy^{*}, \vlambda^{*}\right)$ with an exponent $\Tilde{\mu}$. Then, we draw the following estimations under Assumptions \ref{assum grad f}-\ref{Lip sub path}:

\begin{itemize}
\item [i)] If $\Tilde{\mu} =0$, then the sequence $\left\{\vw_{t}=\left(\vx_{t}, \vy_{t}, \vlambda_{t}\right)\right\}$ converges a.s. in a finite number of steps.
\item [ii)] If $\Tilde{\mu} \in\left(0, \frac{1}{2}\right]$, then there exists $\hat{c}>0$ and $\xi_{3} \in[0,1)$, such that
{	
$$
\left\|\left(\vx_{t}, \vy_{t}, \vlambda_{t}\right)-\left(\vx^{*}, \vy^{*}, \vlambda^{*}\right)\right\| \leq \hat{c}  \xi_{3}^{t}  \quad \text{a.s}.
$$}
\item [iii)] If $\Tilde{\mu} \in\left(\frac{1}{2}, 1\right)$, then there exists $\hat{c}>0$, such that
{	
$$
\left\|\left(\vx_{t}, \vy_{t}, \vlambda_{t}\right)-\left(\vx^{*}, \vy^{*}, \vlambda^{*}\right)\right\| \leq \hat{c} t^{(\Tilde{\mu}-1) /(2 \Tilde{\mu}-1)}  \quad \text{a.s}.
$$}
\end{itemize}
\end{theorem}

\begin{proof}
We also give a sketch of the proof and refer to Appendix \ref{app:thm 3} for more details. This proof is structured by two main steps which contain:
\begin{itemize}
\item First, we define a new sequence and prove it converges a.s. to zero Q-linearly.
\item Second, we conclude that the sequence $\left\{\vw_{t}=\left(\vx_{t}, \vy_{t}, \vlambda_{t}\right)\right\}$ is bounded by the newly defined sequence which converges a.s. Q-linearly, thus the sequence $\left\{\vw_{t}=\left(\vx_{t}, \vy_{t}, \vlambda_{t}\right)\right\}$ converges a.s. to  $\left(\vx^{*}, \vy^{*}, \vlambda^{*}\right)$ R-linearly.
\end{itemize}
\end{proof}

\begin{remark}

We prove in Theorem \ref{thm 3} that our ASVRG-ADMM almost surely achieves R-linear convergence at the rate of $\xi_{3}$ for nonconvex nonsmooth optimization, in accordance with the definition of R-linear convergence presented in Section \ref{sec:intro}.
We construct a potential energy function $\Psi^{s}_{t}$ in (\ref{def of psi}) that satisfies the KL property to ensure convergence when $\Tilde{\mu} \in ( 0,1 / 2]$. This is the first demonstration of linear convergence for a stochastic ADMM incorporating momentum and variance reduction techniques.
\end{remark}

\begin{figure}[h] 
\vspace{-0.45cm}
\centering
\subfigure{\includegraphics[width=0.436\textwidth]{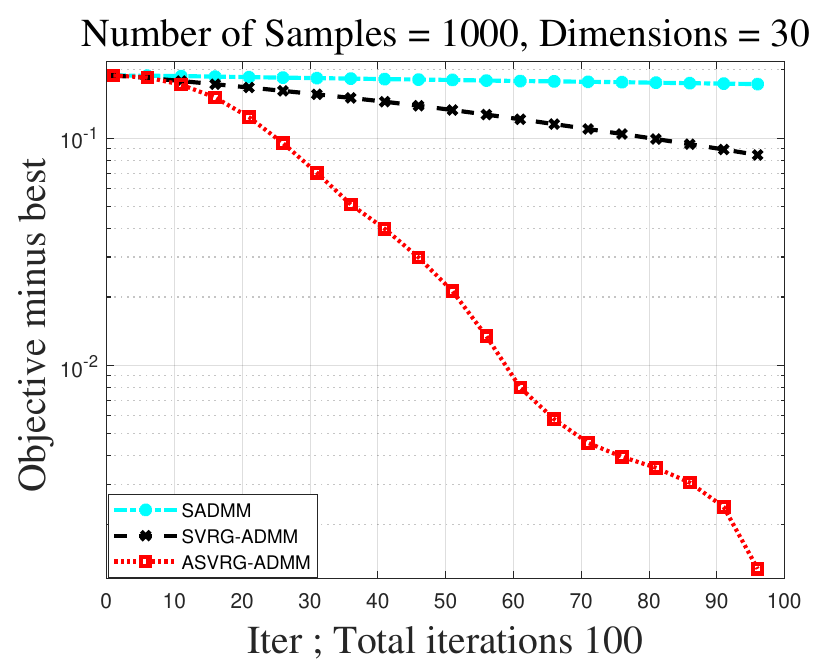}} 
\subfigure{\includegraphics[width=0.432\textwidth]{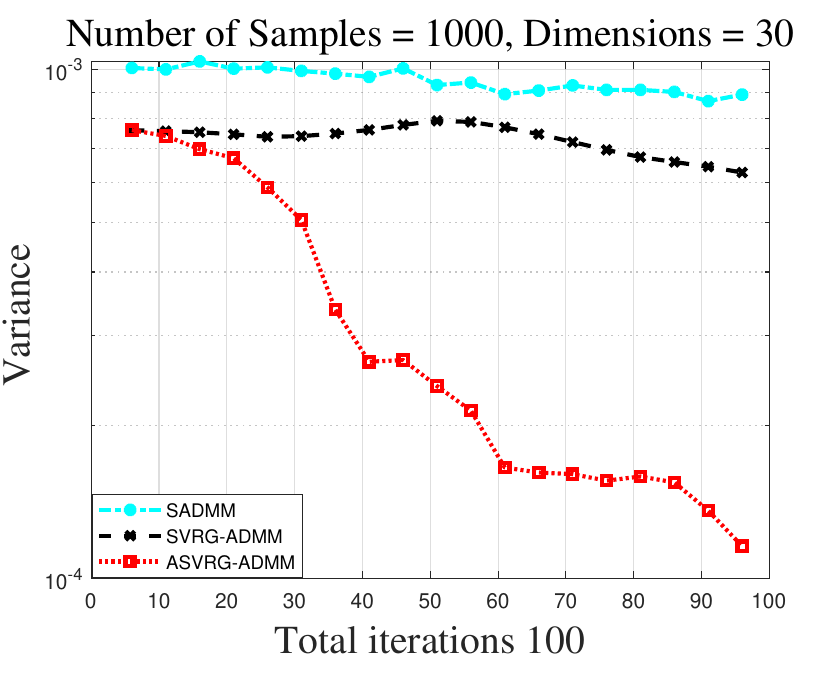}} 
\caption{Comparison of the SADMM, SVRG-ADMM, and
ASVRG-ADMM algorithms for the  nonconvex quadratic problem in the synthetic setting.  The objective values (left) and variances of gradient (right). }
\label{fig_var}
\end{figure}

\section{Numerical Experiments}\label{sec 5}

In this section, we examine the numerical performance of the proposed ASVRG-ADMM algorithm
and present comparison results with the existing methods.
We conducted several different real dataset experiments to validate our theoretical results and show the high efficiency of the implementation of our proposed ASVRG-ADMM algorithmic framework.
We compare ASVRG-ADMM with the following
state-of-the-art methods for nonconvex problems: SADMM in \cite{ouyang2013stochastic}, SVRG-ADMM in \cite{huang2016stochastic}, and SAG-ADMM in \cite{huang2016stochastic}. In the following, all algorithms were
implemented in MATLAB, and the experiments were performed on a PC with an Intel i7-12700F CPU and 16GB memory.


\subsection{Synthetic Data For the Nonconvex Quadratic Problem}


Our setup for the synthetic nonconvex quadratic problem is as follows. We first generate the symmetric matrix $A_i$ from the standard $d_1 \times d_1$ Gaussian distribution. Next, the matrix $A=I$ or $A=[G; I]$, where $G$ is obtained from a sparse inverse
covariance estimation given in  \cite{kim2009multivariate}. Then we solve the following nonconvex quadratic problem:
{	
\begin{equation}
\frac{1}{n} \sum_{i=1}^{n} \frac{\vx^T A_i \vx}{2} + \lambda_1 \|\vy\|_{1},\;\textup{s.t.}\;\vy=A\vx, \nonumber
\end{equation}}
where $\lambda_1= 10^{-4}$ is a positive regularization parameter. Finally, all experimental results are averaged over 300 Monte Carlo experiments.

From Fig.\ref{fig_var}, we can see the  training loss (i.e., the training objective value minus the best value)  and the variance 
of SADMM \cite{ouyang2013stochastic}, SVRG-ADMM \cite{huang2016stochastic}, and ASVRG-ADMM for the nonconvex quadratic problem. All the experimental results show that our ASVRG-ADMM method with smaller variance converges consistently much faster than both SADMM and SVRG-ADMM, which empirically verifies our theoretical results of the effectiveness of the momentum technique resulting in faster convergence.

\subsection{Graph-Guided Fused Lasso}\label{sec 4.1}
\begin{figure}[h] 
\centering
\subfigure{\includegraphics[width=0.432\textwidth]{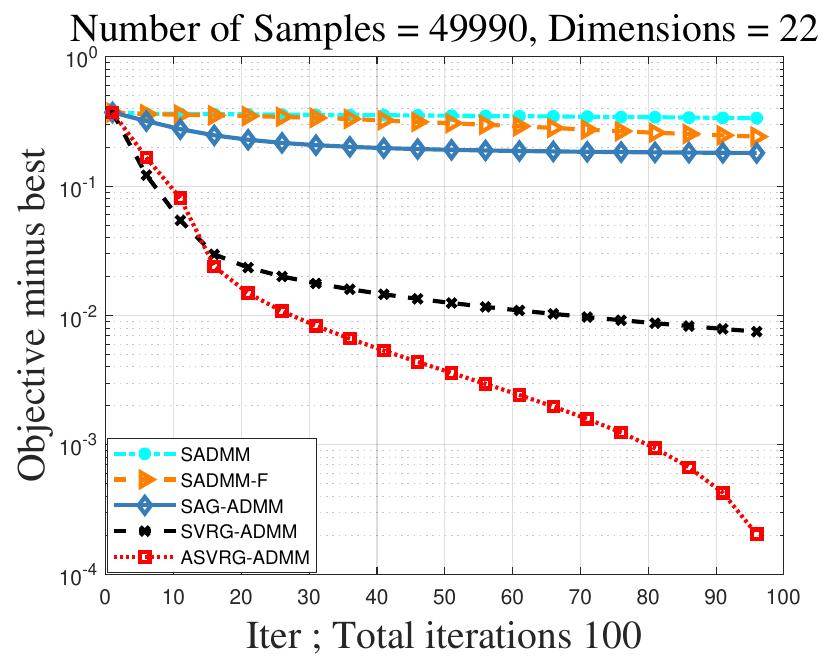} } 
\subfigure{	\includegraphics[width=0.428\textwidth]{ 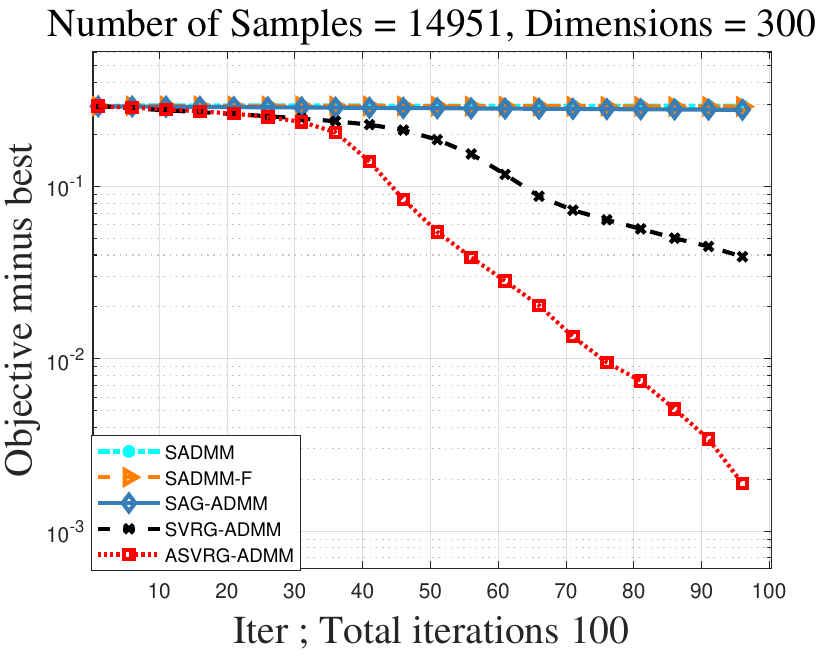}} 
\caption{Comparison of different stochastic ADMM methods
for graph-guided fused lasso problems on the two
data sets: ijcnn1 (left) and  w8a (right).}
\label{fig_3}
\end{figure}

We first evaluate the proposed method for solving the graph-guided fused Lasso problem (\ref{equ fused Lasso}). We set $\eta=2$, $\rho=6$, $\lambda_{1}=10^{-5}$ and $Q=I$, similar to  \cite{huang2016stochastic}.
{Two cases are taken into  account} in Algorithm 1, particularly: a time-varying stepsize parameter
$\eta_{t}=2\sqrt{t+2}$ for the SADMM; a fixed parameter $\eta=2$ for the SADMM-F.
The parameter $m=n$ has been set for the SVRG-ADMM algorithm. At last, the result of all experiments is averaged over 30 Monte Carlos experiments.

We use real datasets ijcnn1, a9a, and w8a downloaded from LIBSVM
to test our proposed method.
The training error of all the methods is shown in Fig.\ref{fig_3}.
The results indicate that
compared to algorithms without variance reduction techniques, e.g., \ SADMM and SADMM-F, the variance reduced stochastic ADMM algorithms (including both SVRG-ADMM and ASVRG-ADMM) converge substantially more quickly.
Notably, { in terms of the experimental convergence performance}, ASVRG-ADMM consistently outperforms all comparison algorithms under all datasets, which is consistent with our theoretical analysis.

\begin{figure}[t] 
\vspace{-0.45cm}
\centering
\subfigure{\includegraphics[width=0.425\textwidth]{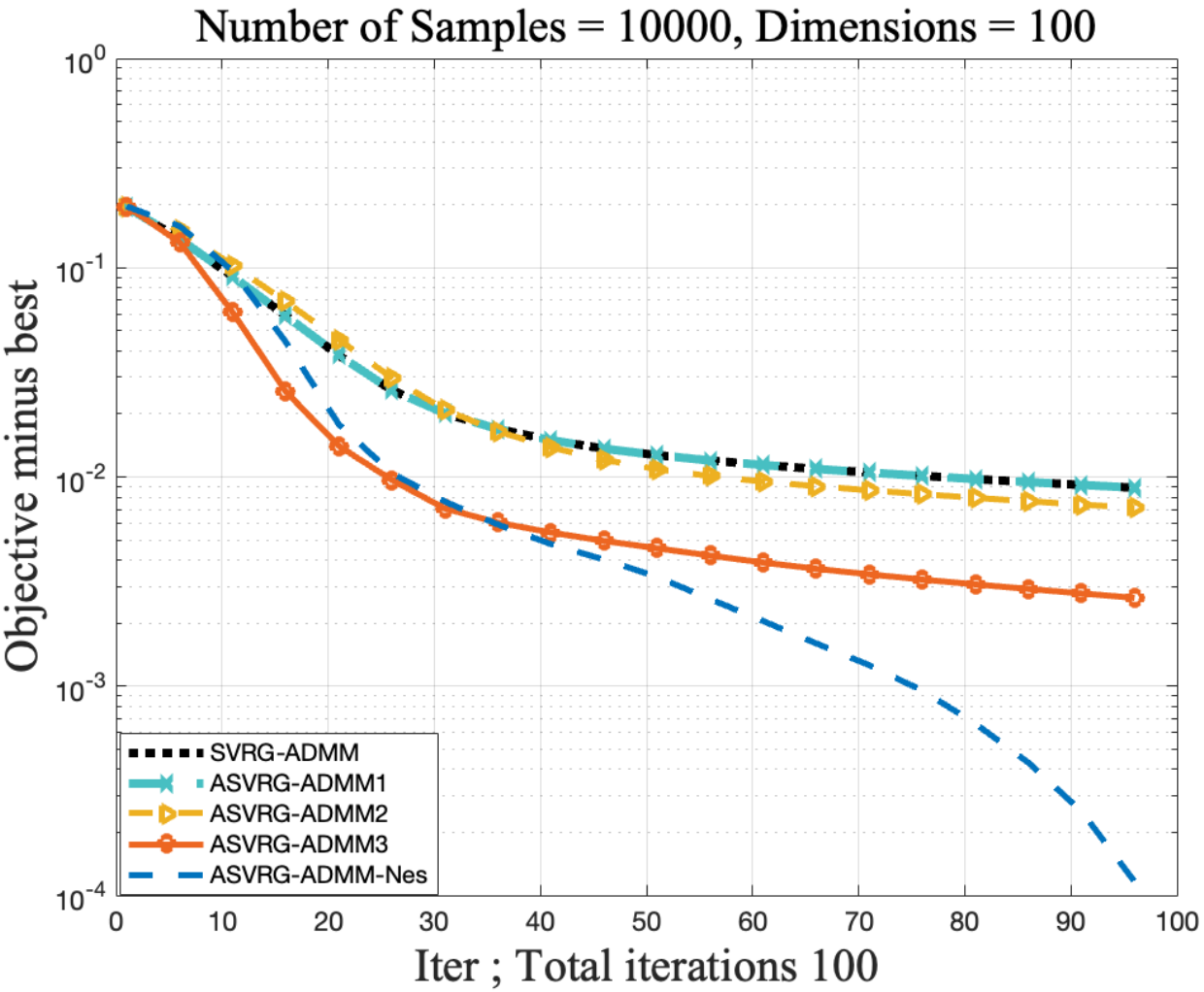}} 
\subfigure{\includegraphics[width=0.435\textwidth]{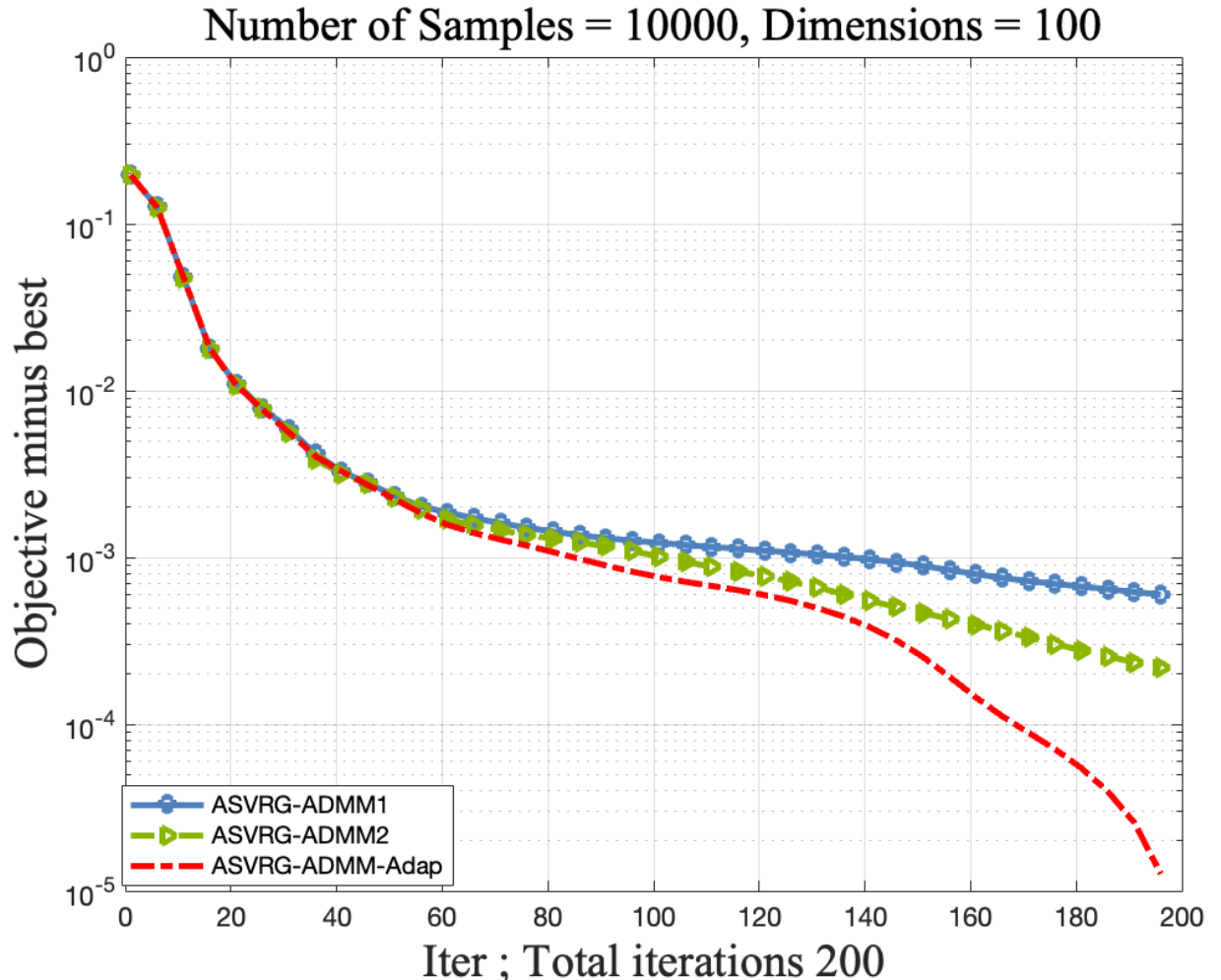}}
\caption{
Comparison of ASVRG-ADMM algorithm for different $\theta$ (left) and $\rho$ (right) on Graph-Guided Fused Lasso using synthetic data.  }
\label{fig_theta and rho}
\end{figure}

We compare various values of momentum $\theta$ in Figure \ref{fig_theta and rho}, namely the larger value $\theta=0.9$ (ASVRG-ADMM2), the smaller but empirically superior value $\theta=0.19$ (ASVRG-ADMM3), and the Nesterov-type momentum parameter \cite{Nesterov2004,liu2020accelerated} (ASVRG-ADMM-Nes).
The results indicate that smaller values of $\theta$ lead to improved experimental outcomes while employing the Nesterov-type decaying momentum parameter results in optimal performance.
Furthermore, we conduct additional comparisons using different values of $\rho$ as shown in Figure 3.
The value of $\rho=60$ in ASVRG-ADMM1 exceeded the experimentally optimal value of $\rho=6$ in ASVRG-ADMM2.
To satisfy the lower bound inequality (\ref{inequ rho}), an adaptive approach known as ASVRG-ADMM-Adap \cite{he2000alternating} can be employed. This approach gradually increases $\rho$ by setting $\rho_{t+1} = \kappa * \rho_{t}$, where $\kappa>1$, to ensure the validity of (\ref{inequ rho}). Smaller values of $\rho$ yield improved experimental performance, with the adaptive approach achieving the best results.

\subsection{Regularized Logistic Regression}
In this subsection, we evaluate the test data performance of ASVRG-ADMM for solving the nonconvex Regularized Logistic Regression (RLR) problem
{	
\begin{equation}\label{acc exp}
\min_{\vx}\frac{1}{n}\sum^{n}_{i=1}\!  f_{i}(\vx)+\lambda_1 \|\vx\|_{1}, \nonumber
\end{equation}}
where $f_{i}(\vx)$ is defined the same as in (\ref{equ fused Lasso}), and the parameter $\lambda_1=2*10^{-5}$. 
In addition, we randomly select 60\%  of the data to train the models and the rest to test the performance.

Fig.\ref{fig_4} shows the experimental results of datasets w8a (left) and a9a (right) respectively. Fig.\ref{fig_4} indicates that both the algorithm SADMM and  SADMM-F consistently converge more slowly than the others and the variances of these two algorithms are larger than the others in all settings, which empirically verifies that variance reduced stochastic ADMM algorithms with smaller variance can perform better.  We observe that SVRG-ADMM and ASVRG-ADMM consistently outperform the others. Moreover, ASVRG-ADMM performs much better than the other methods in all settings, which again verifies the effectiveness of our
ASVRG-ADMM method with the accelerated momentum  technique.

\begin{figure}[htbp] 
\centering
\subfigure{\includegraphics[width=0.435\textwidth]{ 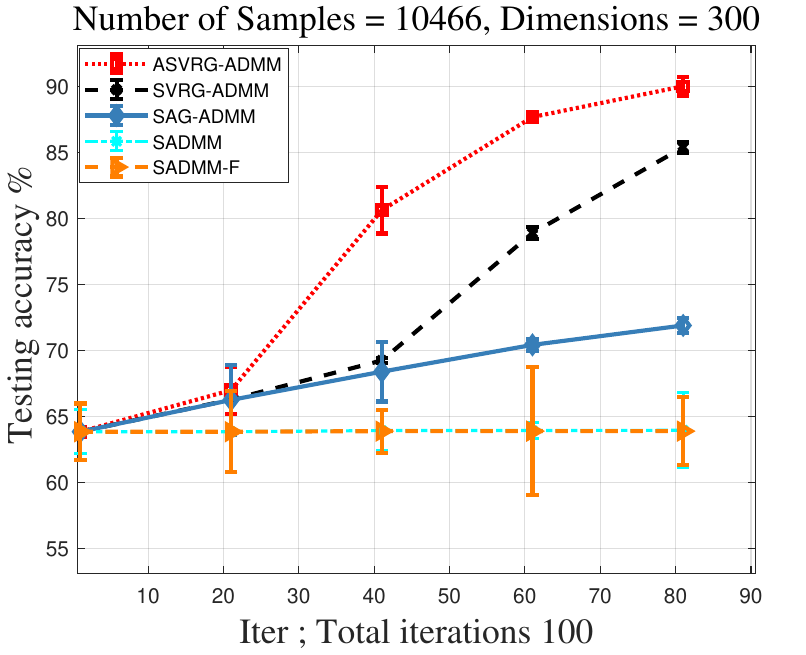}} 
\subfigure{\includegraphics[width=0.43\textwidth]{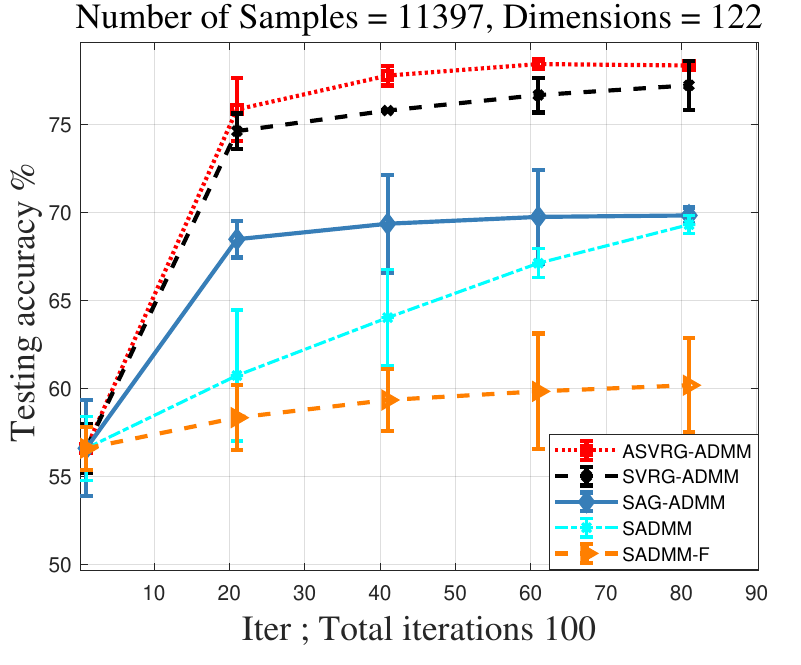}} 
\caption{Comparison of different stochastic ADMM methods
for the accuracy of logistic regression problems.}
\label{fig_4}
\end{figure}

%
%

\section{ Conclusion and Future Work }\label{sec:conclusion}
In this paper,
we integrated both momentum acceleration and the variance reduction trick into stochastic ADMM and proposed an efficient accelerated stochastic variance reduced ADMM for solving the common nonconvex nonsmooth problems.
Furthermore, we theoretically
proved that ASVRG-ADMM obtains a similar convergence to SVRG-ADMM \cite{huang2016stochastic} for nonconvex problems. Especially under the KL property, the convergence rate of ASVRG-ADMM can be improved to the almost sure linear convergence rate. Compared to those of the state-of-the-art stochastic ADMM methods such as SVRG-ADMM,  massive experimental results validate the numerical advantages of  our proposed algorithm.
For future work, there are some interesting directions such as the research of applying the  momentum acceleration trick to other
algorithms like SAG-ADMM  and
SAGA-ADMM \cite{huang2016stochastic}, and the research of the theoretical improvement of the convergence rate of the proposed method.
In addition, it is also interesting to extend ASVRG-ADMM and theoretical results to the distributed optimization and nonsmooth and nonconvex optimization.

\vspace{-0.1cm}

\bibliography{refs}

\newpage

%
%
%



\setcounter{equation}{0}
\renewcommand{\theequation}{S.\arabic{equation}}

\appendices
\section{ Proof of the  Lemma  \ref{lemma of lambda}}\label{app:lemma of h}



\emph{Proof:} \quad	For notational simplicity, we denote
the stochastic gradient $\Delta^{s+1}_t:=\hat{\nabla}f(\vx^{s+1}_t)-\nabla f(\vx^{s+1}_t)$,
where   $\hat{\nabla}f(\vx^{s+1}_t)=\nabla f_{i_t}(\vx^{s+1}_t)-\nabla f_{i_t}(\widetilde{\vx}^{s})+\nabla f(\widetilde{\vx}^{s})$. We also  denote the variance of stochastic gradient $\hat{\nabla}f(\vx^{s+1}_t)$ as $\mathbb{E}_{t}\big[\|\Delta^{s+1}_t\|^2\big]$, 	
and	omit the label $s$ as
$\|\Delta^{s+1}_t\|^2 := \|\Delta_t\|^2$ , $\vx^{s+1}_t := \vx_t$, $\vy^{s+1}_t := \vy_t$, $\vlambda^{s+1}_t := \vlambda_t$, $\widetilde{\vx}^s:=\widetilde{\vx}$, and the conditional expectation operator $\mathbb{E}^s_t:=\mathbb{E}_t$.
By the optimal condition of $\vz$-update (\ref{z-update}) in Algorithm \ref{alg1}, we have
{	
\begin{align}
0  =& \hat{\nabla} f\left(\vx_{t}\right)+\rho A^{\top}\left(A \vz_{t+1}+B \vy_{t+1}-\vc\right)  -A^{\top} \vlambda_{t} + \frac{\theta}{\eta} Q\left(\vz_{t+1}-{ \vz_t}\right) \nonumber \\
= & \hat{\nabla}f(\vx_t)- A^T\vlambda_{t+1} + \frac{\theta}{\eta} Q\left(\vz_{t+1}-{ \vz_t}\right),  \nonumber
\end{align}}
where the second equality follows from the update of $\vlambda$ (\ref{dual-update}) in Algorithm \ref{alg1}. Thus, we have
$A^T\vlambda_{t+1} = \hat{\nabla}f(\vx_t) + \frac{\theta}{\eta} Q\left(\vz_{t+1}-{ \vz_t}\right).$
Recall that
$\vx_{t+1}= \theta \vz_{t+1} + (1-\theta)\widetilde{\vx}$, then it yields $\vx_{t+1}-\vx_t=\theta (\vz_{t+1}-\vz_t)$.
{	
\begin{align} \label{equ 23}
A^T\vlambda_{t+1} = \hat{\nabla}f(\vx_t)- \frac{1}{\eta} Q (\vx_t - \vx_{t+1}).
\end{align}}
Combining (\ref{equ 23}) and the properties of $\|\cdot\|$ 
\begin{align} \label{lambda 24}
	&\|\vlambda_{t+1}-\vlambda_{t}\|^2\nonumber\\
	\leq & (\sigma^{A}_{min})^{-1} \|A^T\vlambda_{t+1}-A^T\vlambda_t\|^2  \nonumber \\
	=& (\sigma^{A}_{min})^{-1}  \|\hat{\nabla}f(\vx_{t})-\hat{\nabla}f(\vx_{t-1}) -\frac{1}{\eta} Q (\vx_t - \vx_{t+1}) -\frac{1}{\eta} Q (\vx_t - \vx_{t-1}) \|^2 \nonumber \\
	=& (\sigma^{A}_{min})^{-1} \|\hat{\nabla}f(\vx_{t})-\nabla f(\vx_{t}) + \nabla f(\vx_{t}) - \nabla f(\vx_{t-1}) + \nabla f(\vx_{t-1}) - \hat{\nabla}f(\vx_{t-1})
	+\frac{1}{\eta} Q (\vx_{t+1} - \vx_{t}) \nonumber\\
	&-\frac{1}{\eta} Q (\vx_t - \vx_{t-1}) \|^2 \nonumber \\
	\mathop{\leq}^{(i)}& \frac{5}{\sigma^{A}_{min}} \|\Delta_t\|^2 +\frac{5 }{\sigma^{A}_{min}} \|\Delta_{t-1}\|^2 +\frac{5 \phi_{\max }^{2}}{\sigma^{A}_{min} \eta^{2}} \left\|\vx_{t}-\vx_{t-1}\right\|^{2}+ \frac{5 \phi_{\max }^{2}}{\sigma^{A}_{min} \eta^{2}} \left\|\vx_{t+1}-\vx_{t}\right\|^{2} \nonumber\\
	&  +\frac{5 }{\sigma^{A}_{min}} \left\| \nabla f(\vx_{t}) - \nabla f(\vx_{t-1})\right\|^{2} ,\nonumber\\
\end{align}
and applying the conditional expectation operator $\mathbb{E}_t$, we obtain
\begin{align}
&\mathbb{E}_{t} \|\vlambda_{t+1}-\vlambda_{t}\|^2\nonumber\\
\mathop{\leq}& \frac{5}{\sigma^{A}_{min}} \|\Delta_t\|^2 +\frac{5 }{\sigma^{A}_{min}} \|\Delta_{t-1}\|^2 +\frac{5 \phi_{\max }^{2}}{\sigma^{A}_{min} \eta^{2}} \left\|\vx_{t}-\vx_{t-1}\right\|^{2}+ \frac{5 \phi_{\max }^{2}}{\sigma^{A}_{min} \eta^{2}}\mathbb{E}_{t} \left\|\vx_{t+1}-\vx_{t}\right\|^{2} \nonumber\\
&  +\frac{5 }{\sigma^{A}_{min}} \left\| \nabla f(\vx_{t}) - \nabla f(\vx_{t-1})\right\|^{2} ,\nonumber
\end{align}
where the first inequality is based on  Assumption \ref{assum4} and the update rule of dual variable given by $\vlambda_{t}^{s+1} - \vlambda_{t+1}^{s+1} = \rho(A \vz_{t+1}^{s+1}+B\vy_{t+1}^{s+1}-c)\in \operatorname{Im}(A)$,
the inequality (i) holds by the Cauchy-Schwartz inequality,  $\phi_{\max}$ denotes the largest eigenvalue of positive matrix $Q$, and $\|Q(x-\vy)\|^2 \leq \phi^2_{\max}\|x-\vy\|^2$.

We  recall the  \cite[Lemma 1]{huang2016stochastic} as follows:
{	\setlength{\abovedisplayskip}{6pt}
\setlength{\belowdisplayskip}{6pt}
\begin{equation}\label{bjc deta}
\mathbb{E}_t \big[ \|\Delta^{s+1}_t\|^2 \big]\leq L^2\|\vx^{s+1}_t-\widetilde{\vx}^{s}\|^2.
\end{equation}}
By inserting (\ref{bjc deta}) and  Assumption \ref{assum grad f} to the above inequality (\ref{lambda 24}), we can estimate the upper bound of $\mathbb{E}_t\big[ \|\vlambda_{t+1}-\vlambda_{t}\|^2\big]$ in \eqref{upp1}.
This completes the proof of Lemma \ref{lemma of lambda}.
\hfill $\blacksquare$

\begin{remark}		
Huang and Chen \cite{huang2018mini} employed a general unbiased stochastic gradient estimator with bounded variance: $$\mathbb{E}\left[\left\|G\left(\vx, \xi_{\mathcal{I}}\right)-\nabla f(\vx)\right\|^2\right] \leq \sigma^2 / M,$$ where $G(\vx, \xi_{\mathcal{I}})=\frac{1}{M} \sum_{i \in \mathcal{I}} G(\vx, \xi_i)$ represents the stochastic gradient estimator, $M$ denotes the mini-batch size, and $\xi_{\mathcal{I}}=\left\{\xi_1, \xi_2, \cdots, \xi_M\right\}$ represents a set of i.i.d. random variables.

Compared to the stochastic gradient estimator with bounded variance in \cite{huang2018mini}, the variance-reduced stochastic gradient defined in Equation (\ref{f-grad-update}) exhibits a decrease in variance with an increasing iteration number, as shown in Equation (\ref{bjc deta}). In contrast, the variance of the general stochastic gradient does not decrease. Our paper leverages the property of variance reduction in the gradient to achieve a linear convergence rate superior to sublinear using the KL property in subsequent analysis.


\end{remark}

\vspace{-0.2cm}
\section{Proof of Lemma \ref{detail lemma of h}} \label{app:lemma 2}
We first define the positive sequences  $\{(h_t^s)_{t=1}^m\}_{s=1}^S$, $\{(\Gamma^{s}_t)_{t=1}^m\}_{s=1}^S$ and the tuple $(\beta_1,\beta_2,\beta_3,\beta_4,\beta_5,\beta_6)$ to be used in constructing the potential energy function and proving the sufficient descent inequality of the potential energy function:
{	\setlength{\abovedisplayskip}{10pt}
\setlength{\belowdisplayskip}{10pt}
\begin{equation}\label{def of h}
h^{s}_t= \left\{
\begin{aligned}
& (2+\alpha_1)h^{s}_{t+1} + \left[ \beta_3+(1+\alpha_1) \beta_1 \right] ,  1 \leq t \leq m-1, \\
& \frac{5L^2}{\sigma^{A}_{min}}\big(\frac{2}{\rho} + \frac{1}{2l_{1}}\big), \quad t=m,
\end{aligned}
\right.\end{equation}}	
{	\setlength{\abovedisplayskip}{1pt}
\setlength{\belowdisplayskip}{1pt}
\begin{equation} \label{def of gamma}
\Gamma^{s}_t =\left\{
\begin{aligned}
&\beta_2-\beta_5-(h_{t+1}^{s}+\beta_1)\Big(1+\frac{1}{\alpha_1}\Big), \ 1 \leq t \leq m-1,  \\
&\beta_6 - \beta_5 - h_{1}^{s}, \quad t=m.
\end{aligned}
\right.\end{equation}}

{	\setlength{\abovedisplayskip}{1pt}
\setlength{\belowdisplayskip}{1pt}
\begin{align} \label{equ defg}
\nonumber \beta_1 &= \left(\frac{\rho+l_{1}}{2}\right) \left(\frac{1-\theta}{\theta}\right)^{2}\sigma^{A}_{max},\\
\nonumber\beta_2 &= \frac{\phi_{\min}}{\eta} -\frac{L}{2} - \frac{5\phi^2_{\max}  }{\sigma^{A}_{min}\eta^2}(\frac{1}{\rho} + \frac{1}{2l_{1}})  + \frac{\rho\sigma^{A}_{min}(l_{2}-(1-\theta))}{2\theta^2 l_{2}},\nonumber\\
\nonumber\beta_3 &= -\frac{\rho\sigma^{A}_{min} }{2}\left[\left(\frac{1-\theta}{\theta}\right)^{2}-\frac{(1-\theta) l_{2}}{\theta^{2}}\right] +\left(\frac{1}{\rho}+\frac{1}{2 l_{1}}\right) \frac{5 L^{2}}{\sigma^{A}_{min} },\nonumber\\
\beta_4 &=  \left(\frac{1}{\rho}+\frac{1}{2 l_{1}}\right) \frac{5 L^{2}}{\sigma^{A}_{min} },\\
\nonumber \beta_{5}&=\left(\frac{1}{\rho}+\frac{1}{2 l_{1}}\right) \frac{5 L^{2} \eta^{2}+5 \phi_{\max }^2}{\sigma^{A}_{min} \eta^{2}},\\
\nonumber\beta_6 &= \frac{\phi_{\min }}{\eta }-\frac{L}{2}-\left(\frac{1}{\rho}+\frac{1}{2 l_{1}}\right) \frac{5 \phi_{\max }^{2}}{\sigma^{A}_{min} \eta^{2}}  +\frac{\sigma^{A}_{min}\rho}{2 \theta^{2}} - \sigma^{A}_{max}\left(\frac{\rho +l_{1}}{2}\right)\left(\frac{1-\theta}{\theta}\right)^{2},
\end{align}}

\emph{Proof:}
This proof is structured by two main parts which contain:
\begin{itemize}
\item
First, we will prove that   $\{(\Psi^{s}_{t})_{t=1}^m\}_{s=1}^S$
is sufficiently and  monotonically {decreasing} over $t\in \{1,2,\cdots,m\}$ in each iteration $s\in \{1,2,\cdots,S\}$.
\item  Second, we will prove that $\Psi^{s}_{m} \geq \mathbb{E}_{0}\Psi^{s+1}_{1}$ for any $s\in \{1,2,\cdots,S\}$.
\end{itemize}

For notational simplicity, we   omit the { superscript} $s$  in the first part, i.e.,
let
$
\vx^{s+1}_t = \vx_t, ~\vy^{s+1}_t = \vy_t, ~
\vlambda^{s+1}_t = \vlambda_t, ~
\widetilde{\vx}^{s}=\widetilde{\vx}.
$
By the  $\vy$-update (\ref{y-update}) in Algorithm \ref{alg1}, we have
{	
\begin{align}\label{equ 26}
\mathcal {L}_\rho (\vx_t, \vy_{t+1},\vlambda_t) \leq \mathcal {L}_\rho (\vx_t, \vy_{t},\vlambda_t).
\end{align}}
The optimal condition of $\vz$-update (\ref{z-update}) in Algorithm \ref{alg1} implies
{	
\begin{align}\label{equ 2222}
0 =&({ \vx_t-\vx_{t+1}})^T\big[\hat{\nabla}f(\vx_t)-A^T\vlambda_t +\rho A^T(A \vz_{t+1}+B\vy_{t+1}-\vc) + \frac{\theta}{\eta} Q\left(\vz_{t+1}-\vz_{t}\right)\big] \nonumber \\
=& (\vx_t-\vx_{t+1})^T\big[\hat{\nabla}f(\vx_t) - \nabla f(\vx_t) + \nabla f(\vx_t) -A^T\vlambda_{t}  \big]  +\rho A^T(A \vz_{t+1}+B\vy_{t+1}-\vc)\nonumber\\ &-(\vx_t-\vx_{t+1})^T{ \frac1\eta} Q(\vx_t-\vx_{t+1})\nonumber \\
\mathop{\leq}& f(\vx_t) - f(\vx_{t+1}) + (\vx_t -\vx_{t+1})^T(\hat{\nabla}f(\vx_t) - \nabla f(\vx_t))  + \frac{L}{2} \|\vx_{t+1}-\vx_t\|^2 - \vlambda_t^T(A \vx_t-A\vx_{t+1}) \nonumber \\
& - \frac{1}{\eta}\|\vx_{t+1} -\vx_t\|^2_Q  + \rho (A\vx_t -A\vx_{t+1})^T(A\vz_{t+1}+B\vy_{t+1}-\vc), \nonumber \\
\end{align}}
where the inequality holds by the Assumption \ref{assum grad f}.  Inserting the equality
$(a-b)^T(c-d) = \frac{1}{2}(\|a-d\|^2-\|b-d\|^2 +\|b-c\|^2 -\|c-a\|^2)$ on the term $ \rho (A\vx_t -A\vx_{t+1})^T(A\vz_{t+1}+B\vy_{t+1}-c)$ into the above inequality (\ref{equ 2222}), we have the following trivial inequality
{	
\begin{align}\label{equ 2333}
0   \mathop{\leq}& f(\vx_t) - f(\vx_{t+1}) + (\vx_t -\vx_{t+1})^T(\hat{\nabla}f(\vx_t) - \nabla f(\vx_t))  +\frac{L}{2} \|\vx_{t+1}-\vx_t\|^2  - \frac{1}{\eta}\|\vx_{t+1}-\vx_t\|^2_Q \nonumber \\
&  - \vlambda_t^T(A \vx_t+B\vy_{t+1}-\vc) + \vlambda_t^T(A\vx_{t+1}+B\vy_{t+1}-\vc)\nonumber\\
&+ \frac{\rho}{2} \Big( \left\|A \vx_{t}+B \vy_{t+1}-\vc\right\|^{2} +\left\|A \vz_{t+1}-A \vx_{t+1}\right\|^{2} -\left\|A \vx_{t+1}+B \vy_{t+1}-\vc\right\|^{2} -\left\|A \vz_{t+1}-A \vx_{t}\right\|^{2} \Big) \nonumber \\
= &  \mathcal {L}_\rho (\vx_t, \vy_{t+1},\vlambda_t)- \mathcal {L}_\rho (\vx_{t+1}, \vy_{t+1},\vlambda_t) +(\vx_t -\vx_{t+1})^T(\hat{\nabla}f(\vx_t) - \nabla f(\vx_t)) +\frac{L}{2} \|\vx_{t+1}-\vx_t\|^2 \nonumber \\
&   + \frac{\rho}{2}\left(\left\|A \vz_{t+1}-A \vx_{t+1}\right\|^{2}-\left\|A \vz_{t+1}-A \vx_{t}\right\|^{2}\right) - \frac{1}{\eta}\|\vx_{t+1}-\vx_t\|^2_Q.
\end{align}}

By inserting $-\phi_{\min}\|\vx_{t+1}-\vx_t\|^2 \geq -\|\vx_{t+1}-\vx_t\|^2_Q$ into the above inequality (\ref{equ 2333}), we have
{	
\begin{align}
0 &\mathop{\leq} \mathcal {L}_\rho (\vx_t, \vy_{t+1},\vlambda_t)- \mathcal {L}_\rho (\vx_{t+1}, \vy_{t+1},\vlambda_t) + (\vx_t -\vx_{t+1})^T(\hat{\nabla}f(\vx_t) - \nabla f(\vx_t)) \nonumber \\
&\quad+ \frac{\rho}{2}\left(\left\|A \vz_{t+1}-A \vx_{t+1}\right\|^{2}-\left\|A \vz_{t+1}-A \vx_{t}\right\|^{2}\right)+\frac{L}{2}\left\|\vx_{t+1}-\vx_{t}\right\|^{2} - { \frac{\phi_{\min}}{\eta}\|\vx_{t+1}-\vx_t\|^2}.\label{equ 302}
\end{align}}

Then, applying  conditioned expectation $\mathbb{E}_t$ on information $i_t$ to (\ref{equ 302}), and using $\mathbb{E}_t[\hat{\nabla} f(\vx_{t})]=\nabla f(\vx_{t})$, we have
{	
\begin{align}\label{equ 28}
\mathbb{E}_t &[\mathcal {L}_\rho (\vx_{t+1}, \vy_{t+1},\vlambda_t)] \leq \mathcal {L}_\rho (\vx_t, \vy_{t+1},\vlambda_t) +{ \Big(\frac{L}{2}- \frac{\phi_{\min}}{\eta}\Big)}
\mathbb{E}_t\left\|\vx_{t+1}-\vx_{t}\right\|^2 \nonumber\\
&+\frac{\rho}{2}\left(\mathbb{E}_t\left\|A \vz_{t+1}-A \vx_{t+1}\right\|^{2}-\mathbb{E}_t\left\|A \vz_{t+1}-A \vx_{t}\right\|^{2} \right).
\end{align}}
By the $\vlambda$-update (\ref{dual-update}) in Algorithm \ref{alg1}, and applying  conditioned expectation $\mathbb{E}_t$ on information $i_t$ again, we have
{	
\begin{align}
&\mathbb{E}_t [\mathcal {L}_\rho (\vx_{t+1}, \vy_{t+1},\vlambda_{t+1})-\mathcal {L}_\rho (\vx_{t+1}, \vy_{t+1},\vlambda_t)] \nonumber\\
=&\frac{1}{\rho}\mathbb{E}_t \|\vlambda_{t+1}-\vlambda_t\|^2 + \mathbb{E}_t\left\langle\vlambda_{t}-\vlambda_{t+1}, A \vx_{t+1}-A \vz_{t+1}\right\rangle. \label{equ 292}
\end{align}}

Combine (\ref{equ 26}), (\ref{equ 28}) and   (\ref{equ 292}) to get
{	
\begin{align}\label{equ 30}
\mathbb{E}_t [\mathcal {L}_\rho (\vx_{t+1}, \vy_{t+1},\vlambda_{t+1})] \leq & \mathcal {L}_\rho (\vx_t, \vy_{t},\vlambda_t) +{ \Big(\frac{L}{2}- \frac{\phi_{\min}}{\eta}\Big)}
\mathbb{E}_t\left\|\vx_{t+1}-\vx_{t}\right\|^2\nonumber\\
&+\frac{\rho}{2}\left(\mathbb{E}_t \left\|A \vz_{t+1}-A \vx_{t+1}\right\|^{2}-\mathbb{E}_t\left\|A \vz_{t+1}-A \vx_{t}\right\|^{2}\right) \nonumber\\
&+\frac{1}{\rho}\mathbb{E}_t\left\|\vlambda_{t}-\vlambda_{t+1}\right\| ^{2}+\mathbb{E}_t\left\langle\vlambda_{t}-\vlambda_{t+1}, A \vx_{t+1}-A \vz_{t+1}\right\rangle.
\end{align}}

Apply the Cauchy-Schwartz inequality  to the term
$
\left\langle\vlambda_{t}-\vlambda_{t+1}, A \vx_{t+1}-A z_{t+1}\right\rangle$ to drive
{	\setlength{\abovedisplayskip}{5pt}
\setlength{\belowdisplayskip}{5pt}
\begin{align*}
\left\langle\vlambda_{t}-\vlambda_{t+1}, A \vx_{t+1}-A z_{t+1}\right\rangle 
\leq &\frac{1}{2 l_{1}}\|\vlambda_{ t}-\vlambda_{t+1}\|^{2}
+\frac{l_{1}}{2}\left\|A \vx_{t+1}-A z_{t+1}\right \|^{2}, \forall l_{1} > 0.
\end{align*}}

Combining the above inequality with (\ref{equ 30}), we have
{
\begin{align}\label{equ 31}
&\mathbb{E}_t [\mathcal {L}_\rho (\vx_{t+1}, \vy_{t+1},\vlambda_{t+1})]\nonumber\\
\leq & \mathcal {L}_\rho (\vx_t, \vy_{t},\vlambda_t) +{ \Big(\frac{L}{2}- \frac{\phi_{\min}}{\eta}\Big)}
\mathbb{E}_t\left\|\vx_{t+1}-\vx_{t}\right\|^2  +\frac{\rho}{2} \mathbb{E}_t\left(\left\|A \vz_{t+1}-A \vx_{t+1}\right\|^{2}-\left\|A \vz_{t+1}-A \vx_{t}\right\|^{2}\right) \nonumber\\
&  +\left(\frac{1}{\rho}+\frac{1}{2 l_{1}}\right)\mathbb{E}_t \left\|\vlambda_{t}-\vlambda_{t+1}\right\|^{2}+\frac{l_{1}}{2}\mathbb{E}_t\left\| A \vz_{t+1}-A \vx_{t+1} \right\|^{2} \nonumber\\
\mathop{\leq}^{(i)}  &\mathcal {L}_\rho (\vx_t, \vy_{t},\vlambda_t) +\frac{\rho+l_{1}}{2}\mathbb{E}_t\left\|A \vz_{t+1}-A \vx_{t+1}\right\|^{2}-\left[\frac{\phi_{\min }}{\eta}-\frac{L}{2}-\Big(\frac{1}{\rho}+\frac{1}{2 l_{1}}\Big) \frac{5 \phi_{\max }^{2}}{\sigma^{A}_{min} \eta^{2}}\right] \mathbb{E}_t\left\|\vx_{t+1}-\vx_{t}\right\|^{2} \nonumber\\
&   -\frac{\rho}{2}\mathbb{E}_t\left\|A \vz_{t+1}-A \vx_{t}\right\|^{2}
+
\Big(\frac{1}{\rho}+\frac{1}{2 l_{1}}\Big) \big(\frac{5 L^{2}}{\sigma^{A}_{min}}\left\|\vx_{t}-\widetilde{\vx}\right\|^{2}+\frac{5 L^{2}}{\sigma^{A}_{min}}\left\|\vx_{t-1}-\widetilde{\vx}^{2}\right\|^{2}\nonumber\\
&+ \frac{5 L^{2} \eta^{2} +5 \phi_{\max }^{2} }{\sigma^{A}_{min} \eta^{2}}\left\|\vx_{t}-\vx_{t-1}\right\|^{2}
\big),
\end{align}}
where the inequality $(i)$ holds by Lemma \ref{lemma of lambda}.

Notice by the update of $\vx_{t+1}$ in Algorithm \ref{alg1} that $\vz_{t+1}-\vx_{t}=\frac1\theta(\vx_{t+1}-\vx_{t}) +\frac{1-\theta}{\theta}(\widetilde{\vx}-\vx_{t})$ and $\vz_{t+1}-\vx_{t+1}= \frac{1-\theta}{\theta}(\vx_{t+1}-\widetilde{\vx})$.
{Substitute  these two equations into
$\left\|A \vz_{t+1}-A \vx_{t+1}\right\|^{2}$ and $\left\|A \vz_{t+1}-A \vx_{t}\right\|^{2}$ to have
\begin{align}
&\frac{\rho+l_{1}}{2}\left\|A \vz_{t+1}-A \vx_{t+1}\right\|^{2}-\frac{\rho}{2}\left\|A \vz_{t+1}-A \vx_{t}\right\|^{2} \nonumber\\
\leq& \frac{\rho+l_{1}}{2}  \sigma^{A}_{max} \left\|\vz_{t+1}-\vx_{t+1}\right\|^{2}-\frac{\rho}{2} \sigma^{A}_{min}\left\|\vz_{t+1}-\vx_{t}\right\|^{2} \nonumber\\
=&-\frac{\rho \sigma_{\min }^{A}}{2}\left[\frac{1}{\theta^{2}}\left\|\vx_{t+1}-\vx_{t}\right\|^{2}+\left(\frac{1-\theta}{\theta}\right)^{2}\left\|\vx_{t}-\widetilde{\vx}\right\|^{2}\right] - \frac{\rho \sigma^{A}_{min} }{\theta}  \left(1-\frac{1}{\theta}\right) <\vx_{t+1}-\vx_{t}, \vx_{t}-\widetilde{\vx}> \nonumber \nonumber\\
&+ \frac{\rho+l_{1}}{2} \sigma_{\max }^{A}\left(\frac{1-\theta}{\theta}\right)^{2}\left\|\vx_{t+1}-\widetilde{\vx}\right\|^{2}, \nonumber\\
\end{align} }	
then we can drive
\begin{align} \label{euq 32}
&\frac{\rho+l_{1}}{2}\left\|A \vz_{t+1}-A \vx_{t+1}\right\|^{2}-\frac{\rho}{2}\left\|A \vz_{t+1}-A \vx_{t}\right\|^{2} \nonumber\\
\mathop{\leq}^{(i)}& (\frac{\rho+l_{1}}{2}) \left(\frac{1 -\theta}{\theta}\right)^{2} \sigma^{A}_{max}\left\|\vx_{t+1}-\widetilde{\vx}\right\|^{2} -  \frac{\rho \sigma^{A}_{min}}{2} \left[\frac{1}{\theta^{2}}\left\|\vx_{t+1}-\vx_{t}\right\|^{2}+\left(\frac{1-\theta}{\theta}\right)^{2}\left\|\vx_{t}-\widetilde{\vx}\right\|^{2} \right]\nonumber\\
&+\frac{\rho \sigma^{A}_{min}}{2}\frac{(1-\theta)}{\theta^{2} l_{2}}\left\|\vx_{t+1}-\vx_{t}\right\|^{2}+\frac{\rho  \sigma^{A}_{min}}{2} \frac{(1-\theta) l_{2}}{\theta^{2}}\left\|\vx_{t}-\widetilde{\vx}\right\|^{2},
\end{align}
where the first equality holds by the equality $\vx_{t+1}= \theta \vz_{t+1} + (1-\theta)\widetilde{\vx}$ and the inequality (i) holds by using the   Cauchy-Schwartz on the term $<\vx_{t+1}-\vx_{t}, \vx_{t}-\widetilde{\vx}>$.

Thus, by the inequalities (\ref{equ 30}), (\ref{equ 31}) and (\ref{euq 32}), we have
{	
\begin{align}
&\mathbb{E}_t [\mathcal {L}_\rho (\vx_{t+1}, \vy_{t+1},\vlambda_{t+1})] \nonumber\\
\leq  &\mathcal {L}_\rho (\vx_{t}, \vy_{t},\vlambda_{t})
- \big(\frac{\phi_{\min}}{\eta} -\frac{L}{2} - \frac{5\phi^2_{\max}  }{\sigma^{A}_{min}\eta^2}(\frac{1}{\rho} + \frac{1}{2l_{1}})  + \frac{\rho\sigma^{A}_{min}(l_{2}-(1-\theta))}{2\theta^2 l_{2}} \big) \mathbb{E}_t\left\|\vx_{t+1}-\vx_{t}\right\|^{2} \nonumber\\
&+\Bigg[\left(\frac{1}{\rho}+\frac{1}{2 l_{1}}\right) \frac{5 L^{2}}{\sigma^{A}_{min} }- \frac{\rho\sigma^{A}_{min}}{2}\big(\left(\frac{1-\theta}{\theta}\right)^{2} -\frac{(1-\theta) l_{2}}{\theta^{2}}\big) \Bigg] \left\|\vx_{t}-\widetilde{\vx}\right\|^{2} \nonumber\\
&+ \left(\frac{1}{\rho}+\frac{1}{2 l_{1}}\right) \frac{5 L^{2}}{\sigma^{A}_{min} } \left\|\vx_{t-1}-\widetilde{\vx}\right\|^{2} +\left(\frac{1}{\rho}+\frac{1}{2 l_{1}}\right) \frac{5 L^{2} \eta^{2}+5 \phi_{\max }^2}{\sigma^{A}_{min} \eta^{2}}\left\|\vx_{t}-\vx_{t-1}\right\|^{2}\nonumber\\
&+\left(\frac{\rho+l_{1}}{2}\right) \left(\frac{1-\theta}{\theta}\right)^{2}\sigma^{A}_{max}\mathbb{E}_t\left\|\vx_{t+1}-\widetilde{\vx}\right\|^{2}. \label{equ 382}
\end{align}}
The definitions of $\beta_1, \beta_2, \beta_3, \beta_4,$ \text{and}  $\beta_5$ given in  Lemma \ref{detail lemma of h},  (\ref{equ 382}) can be rewritten as
{	
\begin{align}
&\mathbb{E}_t [\mathcal {L}_\rho (\vx_{t+1}, \vy_{t+1},\vlambda_{t+1})] \nonumber\\
\leq &\mathcal {L}_\rho (\vx_{t}, \vy_{t},\vlambda_{t})  - \beta_2 \mathbb{E}_t\left\|\vx_{t+1}-\vx_{t}\right\|^{2}+\beta_3\left\|\vx_{t}-\widetilde{\vx}\right\|^{2}
+\beta_4 \left\|\vx_{t-1}-\widetilde{\vx}\right\|^{2} \nonumber\\
&+ \beta_5\left\|\vx_{t}-\vx_{t-1}\right\|^{2}+\beta_1\mathbb{E}_t\left\|\vx_{t+1}-\widetilde{\vx}\right\|^{2}.\nonumber
\end{align}}

For notational simplicity, we let $\mathcal {L}_\rho(t) := \mathcal {L}_\rho (\vx_{t}, \vy_{t},\vlambda_{t})$, $\mathcal {L}_\rho(t+1) := \mathbb{E}_t\mathcal {L}_\rho (\vx_{t+1}, \vy_{t+1},\vlambda_{t+1})$ then we have
{	
\begin{align}
&\mathcal{L}_{\rho}(t+1)+\beta_{5}\mathbb{E}_t\left\|\vx_{t+1}-\vx_{t}\right\|^{2}+h\left(\mathbb{E}_t\left\|\vx_{t+1}-\widetilde{\vx}\right\|^{2}+\left\|\vx_{t}-\widetilde{\vx}\right\|^{2}\right) \nonumber\\
\leq& \mathcal {L}_\rho (t)-(\beta_2-\beta_5)\mathbb{E}_t\left\|\vx_{t+1}-\vx_{t}\right\|^{2}+(h+\beta_1)\mathbb{E}_t\left\|\vx_{t+1}-\widetilde{\vx}\right\|^{2} +(h+\beta_3) \left\|\vx_{t}-\widetilde{\vx}^{2}\right\|^{2}\nonumber\\
&+ \beta_4\left\|\vx_{t-1}-\widetilde{\vx}^{2}\right\|^{2}+\beta_5\left\|\vx_{t}-\vx_{t-1}\right\|^{2}.  \label{equ 399}
\end{align}}
By the Cauchy-Schwartz inequality again, for any $\alpha_1 > 0$ we have
{	
\begin{align}
(h+\beta_1)\left\|\vx_{t+1}-\widetilde{\vx}\right\|^{2} 
\leq &(h+\beta_1 )\left\{\Big(1+\frac{1}{\alpha_1}\Big)\left\|\vx_{t+1}-\vx_{t}\right\|^{2}+\big(1+\alpha_1 \big)\left\|\vx_{t}-\widetilde{\vx}\right\|^{2}\right\}. \nonumber
\end{align}}

Plugging it into (\ref{equ 399}) is to derive
{	
\begin{align}\label{equ 366}
&
\mathcal {L}_\rho (t+1)+\beta_5 \mathbb{E}_t\left\|\vx_{t+1}-\vx_{t}\right\|^{2}+ h\big( \mathbb{E}_t\left\|\vx_{t+1}-\widetilde{\vx}\right\|^{2}+\left\| \vx_{t}-\widetilde{\vx} \right\|^{2} \big) \nonumber\\
\leq & \mathcal {L}_\rho (t) + \beta_5\left\|\vx_{t}-\vx_{t-1}\right\|^{2} - \big[\beta_2-\beta_5-(h+\beta_1) \big(1+\frac{1}{\alpha_1}\big)\big]\mathbb{E}_{t}\left\|\vx_{t+1}-\vx_{t}\right\|^{2} \nonumber\\
&-\left[h+\beta_3-\beta_4+\left(1+\alpha_1\right)(h+\beta_1)\right]\left\|\vx_{t-1}-\widetilde{\vx}\right\|^{2}\nonumber\\
&+ \left[h+ \beta_3+\left(1+\alpha_1 \right)(h+\beta_1)\right]\Big(\left\|\vx_{t}-\widetilde{\vx}\right\|^{2}+\left\|\vx_{t-1}-\widetilde{\vx}\right\|^{2}\Big) .
\end{align}}

Using the definition of  $\{(\Psi_t^s)_{t=1}^m\}_{s=1}^S$, $h = h^s_{t+1}$ and inserting both (\ref{def of h}) and (\ref{def of gamma}) into (\ref{equ 366}), we have
{	
\begin{align}\label{equ 35}
\mathbb{E}_{t}\Psi^{s+1}_{t+1} \leq \Psi^{s+1}_{t} - \big[&h^{s+1}_{t+1} +\beta_3-\beta_4 +\left(1+\alpha_1\right) (h^{s+1}_{t+1}+\beta_1)\big] \|\vx^{s+1}_{t-1}-\widetilde{\vx}^{s}\|^2  -\Gamma^{s+1}_t \mathbb{E}_{t}\|\vx^{s+1}_{t+1}-\vx^{s+1}_t\|^2,
\end{align}}
for any $s\in \{0,1,\cdots,S-1\}$.
Since $\Gamma^s_t >0, \ \forall t\in \{1,2,\cdots, m\}$, we have proved the first part.

In the following, we will prove the second part. We begin by estimating the upper bound of $\mathbb{E}_{0}\|\vlambda^{s+1}_0-\vlambda^{s+1}_1\|^2$.

Since $\vlambda^{s+1}_0 = \vlambda^s_m$, $\vy^{s+1}_0=\vy^s_m$ and $\vx^{s+1}_0=\vx^s_m=\widetilde{\vx}^s$, we have
{	
\begin{align}\label{equ 36}
&\mathbb{E}_{0}\|\vlambda^{s+1}_0-\vlambda^{s+1}_1\|^2  = \mathbb{E}_{0}\|\vlambda^{s}_m-\vlambda^{s+1}_1\|^2 \nonumber\\
\leq &(\sigma^{A}_{min})^{-1} \mathbb{E}_{0}\|A^T\vlambda^{s}_m-A^T\vlambda^{s+1}_1\|^2 \nonumber \\
\mathop{=}^{(i)}& (\sigma^{A}_{min})^{-1} \mathbb{E}_{0}\|\hat{\nabla}f(\vx^s_{m-1})-\hat{\nabla}f(\vx^{s+1}_{0}) -\frac{1}{\eta} Q(\vx^s_{m-1}-\vx^s_m)
-\frac{1}{\eta} Q (\vx^{s+1}_{0}-\vx^{s+1}_1)\|^2  \nonumber \\
\mathop{=}^{(ii)} &(\sigma^{A}_{min})^{-1} \mathbb{E}_{0}\|\hat{\nabla}f(\vx^s_{m-1})
- \nabla f(\vx^s_{m-1}) + \nabla f(\vx^s_{m-1})-\nabla f(\vx^{s}_{m})-\frac{1}{\eta} Q(\vx^s_{m-1}-\vx^s_m)\nonumber\\
& - \frac{1}{\eta} Q(\vx^{s+1}_{0}-\vx^{s+1}_1)\|^2  \nonumber \\
\mathop{\leq}^{(iii)} &\frac{5L^2}{\sigma^{A}_{min}} \|\vx^s_{m-1}-\widetilde{\vx}^{s-1}\|^2 + \frac{5 L^{2} \eta^{2}+5 \phi_{\text {max }}^{2}}{\sigma^{A}_{min} \eta^{2}} \|\vx^s_{m-1}-\vx^s_{m}\|^2  + \frac{5\phi^2_{\max}}{\sigma^{A}_{min}\eta^2} \mathbb{E}_{0}\|\vx^{s+1}_{0}-\vx^{s+1}_1\|^2,
\end{align}}
where the equality $(i)$ holds by the equality (\ref{equ 23}), the inequality $(iii)$  holds by  Assumption \ref{assum grad f} and the (\ref{bjc deta}), and the equality $(ii)$ holds by the following result:
{	\setlength{\abovedisplayskip}{2pt}
\setlength{\belowdisplayskip}{2pt}
\begin{align}
\hat{\nabla}f(\vx^{s+1}_{0}) & = \nabla f_{i_t}(\vx^{s+1}_{0}) - \nabla f_{i_t}(\widetilde{\vx}^{s}) + \nabla f(\widetilde{\vx}^{s})  \nonumber\\
& = \nabla f_{i_t}(\vx^{s}_{m}) - \nabla f_{i_t}(\vx^{s}_m) + \nabla f(\vx^{s}_m)  \nonumber\\
& = \nabla f(\vx^{s}_m). \nonumber
\end{align}}

By (\ref{equ 26}), we have
{	
\begin{align}\label{equ 37}
\mathcal {L}_\rho (\vx^{s+1}_0, \vy^{s+1}_{1},\vlambda^{s+1}_0)
\leq &\mathcal {L}_\rho (\vx^{s+1}_0, \vy^{s+1}_{0},\vlambda^{s+1}_0)\nonumber\\
= &\mathcal {L}_\rho (\vx^{s}_m, \vy^{s}_{m},\vlambda^{s}_m).
\end{align}}
Similarly, using (\ref{equ 28}), we have
{	
\begin{align}\label{equ 38}
&\mathbb{E}_{0} [\mathcal {L}_\rho (\vx^{s+1}_{1}, \vy^{s+1}_{1},\vlambda^{s+1}_0)] \leq \mathcal {L}_\rho (\vx^{s+1}_0, \vy^{s+1}_{1},\vlambda^{s+1}_0)+{ \Big(\frac{L}{2}- \frac{\phi_{\min}}{\eta}\Big)} \mathbb{E}_{0}\left\|\vx^{s+1}_{1}-\vx^{s+1}_{0}\right\|^{2}\nonumber\\
&  +\frac{\rho}{2}\left(\mathbb{E}_{0}\left\|A \vz^{s+1}_{1}-A \vx^{s+1}_{1}\right\|^{2}-\mathbb{E}_{0}\left\|A \vz^{s+1}_{1}-A \vx^{s+1}_{0}\right\|^{2}\right).
\end{align}

For the update of $\vlambda$ (\ref{dual-update}):
\begin{align}\label{equ 39}
&\mathbb{E}_{0} [\mathcal {L}_\rho\left(\vx_{1}^{s+1}, \vy_{1}^{s+1}, \vlambda_1^{s+1}\right)] \nonumber\\
\leq& \mathbb{E}_{0} [\mathcal {L}_\rho\left(\vx_{1}^{s+1}, \vy_{1}^{s+1}, \vlambda_{0}^{s+1}\right)] +\frac{l_{1}}{2} \mathbb{E}_{0}\left\|A \vz^{s+1}_{1}-A \vx^{s+1}_{1}\right\|^{2}+\left(\frac{1}{\rho}+\frac{1}{2 l_{1}}\right)\mathbb{E}_{0}\left\|\vlambda^{s+1}_{1}-\vlambda^{s+1}_{ 0}\right\|^{2},
\end{align}
where the inequality holds by the Cauchy-Schwartz inequality just like in (\ref{equ 30}).
Combine (\ref{equ 37}), (\ref{equ 38}) with (\ref{equ 39}) and use  similar tricks, we have}
{	
\begin{align}\label{equ 47}
& \mathbb{E}_{0} [\mathcal {L}_\rho\left(\vx_{1}^{s+1}, \vy_{1}^{s+1}, \vlambda_1^{s+1}\right)] \nonumber\\
\leq & \mathbb{E}_{0} \Bigg[\mathcal {L}_\rho\left(\vx_{0}^{s+1}, \vy_{0}^{s+1}, \vlambda_{0}^{s+1}\right) +{ \Big(\frac{L}{2}- \frac{\phi_{\min}}{\eta}\Big)}\left\|\vx^{s+1}_{1}-\vx^{s+1}_{0}\right\|^{2}  + \left(\frac{1}{\rho}+\frac{1}{2 l_{1}}\right)\left\|\vlambda^{s+1}_{0}-\vlambda^{s+1}_{1}\right\|^{2} \nonumber\\
&+ \frac{\rho+l_{1}}{2}\left\|A \vz^{s+1}_{1}-A  \vx^{s+1}_{1}\right\|^{2}-\frac{\rho }{2}\left\|A \vz^{s+1}_{1}-A \vx^{s+1}_{0}\right\|^{2} \Bigg] \nonumber\\
\mathop{\leq }^{(i)}& - \big[ \frac{\phi_{\min }}{\eta }-\frac{L}{2}-\left(\frac{1}{\rho}+\frac{1}{2 l_{1}}\right) \frac{5 \phi_{\max }^{2}}{\sigma^{A}_{min} \eta^{2}} +\frac{\sigma^{A}_{min}\rho}{2 \theta^{2}} - \left(\frac{\rho +l_{1}}{2}\right)\left(\frac{1-\theta}{\theta} \right)^{2}\sigma^{A}_{max} \big]\mathbb{E}_{0} \left\|\vx^{s+1}_{1}-\vx^{s+1}_{0}\right\|^{2} \nonumber\\
& +\mathcal {L}_\rho\left(\vx_{0}^{s+1}, \vy_{0}^{s+1}, \vlambda_{0}^{s+1}\right)+ \left(\frac{1}{\rho}+\frac{1}{2 l_{1}}\right) \frac{5 L^{2}}{\sigma^{A}_{min} }\left\|\vx_{m-1}^{s}-\widetilde{\vx}^{s-1}\right\|^{2}\nonumber\\ &+\left(\frac{1}{\rho}+\frac{1}{2 l_{1}}\right) \frac{5 L^{2} \eta^{2}+5 \phi_{\max }^2}{\sigma^{A}_{min} \eta^{2}}\left\|\vx_{m-1}^{s}-{\vx}_{m}^{s}\right\|^{2} \nonumber\\
\leq & \mathcal {L}_\rho\left(\vx_{0}^{s+1}, \vy_{0}^{s+1}, \vlambda_{0}^{s+1}\right) - \beta_6\mathbb{E}_{0} \left\|\vx^{s+1}_{1}-\vx^{s+1}_{0}\right\|^{2} + \beta_4\left\|\vx_{m-1}^{s}-\widetilde{\vx}^{s-1}\right\|^{2} +\beta_5 \left\|\vx_{m-1}^{s}-{\vx}_{m}^{s}\right\|^{2},
\end{align}}
where $\beta_6$ is given in  Lemma \ref{detail lemma of h}, the equality $(i)$ holds by  using the (\ref{equ 36}), the equality $\vx^{s+1}_{t+1}= \theta \vz^{s+1}_{t+1} + (1-\theta)\widetilde{\vx}^s$ and the following  Cauchy-Schwartz  inequality for the last two terms of the first inequality:
\begin{align}
\frac{1-\theta}{\theta^2}\langle \vx_1^{s+1}-\vx_0^{s+1}, \widetilde{\vx}^s-\vx_0^{s+1}\rangle 
\leq &\frac{1}{2\theta^2} \Big(\|\vx_1^{s+1}-\vx_0^{s+1}\|^2+(1-\theta)^2\|\widetilde{\vx}^s-\vx_0^{s+1}\|^2\Big). \nonumber
\end{align}
Also, for  notational simplicity, let $\mathcal {L}_\rho(0) := \mathcal {L}_\rho (\vx_{0}, \vy_{0},\vlambda_{0})$, $\mathcal {L}_\rho(1) := \mathbb{E}_{0}\mathcal {L}_\rho (\vx_{1}, \vy_{1},\vlambda_{1})$, and sum
\begin{equation}
\beta_5\left\|\mathbf{x}_1^{s+1}-\mathbf{x}_0^{s+1}\right\|^2+h_1^{s+1}\left[\mathbb{E}_{0}\left\|\mathbf{x}_1^{s+1}-\widetilde{\mathbf{x}}^s\right\|^2+\left\|\mathbf{x}_0^{s+1}-\widetilde{\mathbf{x}}^s\right\|^2\right]\nonumber
\end{equation}
to both sides of the (\ref{equ 47}) then we have
\begin{equation}
\begin{aligned}
& \mathcal{L}_\rho(1)+\beta_5\mathbb{E}_{0}\left\|\mathbf{x}_1^{s+1}-\mathbf{x}_0^{s+1}\right\|^2+h_1^{s+1}[\mathbb{E}_{0}\left\|\mathbf{x}_1^{s+1}-\widetilde{\mathbf{x}}^s\right\|^2+\left\|\mathbf{x}_0^{s+1}-\widetilde{\mathbf{x}}^s\right\|^2] \\
\leq & \mathcal{L}_\rho(0)-\beta_6\mathbb{E}_{0}\left\|\mathbf{x}_1^{s+1}-\mathbf{x}_0^{s+1}\right\|^2+\beta_4\left\|\mathbf{x}_{m-1}^s-\widetilde{\mathbf{x}}^{s-1}\right\|^2+\beta_5\left\|\mathbf{x}_{m-1}^s-\mathbf{x}_m^s\right\|^2+\beta_5\mathbb{E}_{0}\left\|\mathbf{x}_1^{s+1}-\mathbf{x}_0^{s+1}\right\|^2 \\
& +h_1^{s+1}\left[\mathbb{E}_{0}\left\|\mathbf{x}_1^{s+1}-\widetilde{\mathbf{x}}^s\right\|^2+\left\|\mathbf{x}_0^{s+1}-\widetilde{\mathbf{x}}^s\right\|^2\right] \\
= & \mathcal{L}_\rho(0)-\left(\beta_6-\beta_5-h_1^{s+1}\right)\mathbb{E}_{0}\left\|\mathbf{x}_1^{s+1}-\mathbf{x}_0^{s+1}\right\|^2+h_1^{s+1}\left[\left\|\mathbf{x}_m^s-\widetilde{\mathbf{x}}^{s-1}\right\|^2+\left\|\mathbf{x}_{m-1}^s-\widetilde{\mathbf{x}}^{s-1}\right\|^2\right] \\
& +\beta_5\left\|\mathbf{x}_{m-1}^s-\mathbf{x}_m^s\right\|^2-\left(h_1^{s+1}-\beta_4\right)\left\|\mathbf{x}_{m-1}^s-\widetilde{\mathbf{x}}^{s-1}\right\|^2-h_1^{s+1}\left\|\mathbf{x}_m^s-\widetilde{\mathbf{x}}^{s-1}\right\|^2.
\end{aligned}\nonumber
\end{equation}
By the notation $h_{1}^{s+1} = \left(\frac{2}{\rho}+\frac{1}{2 l_{1}}\right) \frac{5 L^{2}}{\sigma^{A}_{min}}$ in (\ref{def of h}) and the definition of the sequence $\{(\Psi_t^s)_{t=1}^m\}_{s=1}^S$, we have
{	
\begin{align}\label{equ 42}
\mathbb{E}_{0}\Psi^{s+1}_{1} \leq \Psi^{s}_{m} &- \Gamma^{s}_m\mathbb{E}_{0}\|\vx^{s+1}_1 - \vx^{s+1}_{0}\|^2 -\frac{5L^2}{\sigma^{A}_{min}\rho} \|\vx^s_{m-1}-\widetilde{\vx}^{s-1}\|^2.
\end{align}}
Since $\Gamma^{s}_m>0,\ \forall s \geq 1$, we can obtain the above result of the second part. 

Thus, we prove the above conclusion.
\hfill $\blacksquare$

\section{Proof of  Theorem \ref{thm 1}}
\label{app:theorem 1}
\emph{Proof:} \quad Using the above proofs, inequalities (\ref{equ 35}) and (\ref{equ 42}), we have
{	\setlength{\abovedisplayskip}{10pt}
\setlength{\belowdisplayskip}{10pt}
\begin{align}\label{equ 44}
\mathbb{E}_{t}^{s+1}[\Psi^{s+1}_{t+1}]\leq &\Psi^{s+1}_{t}-\Gamma^{s+1}_t \mathbb{E}_{t}^{s+1}[\|\vx^{s+1}_{t+1}-\vx^{s+1}_t\|^2] - \left[h^{s+1}_{t+1}+\left(1+\alpha_1\right)(h^{s+1}_{t+1}+\beta_1)\right] \|\vx^{s+1}_{t-1}-\widetilde{\vx}^{s}\|^2.
\end{align}}
{The above inequality requires $\beta_3-\beta_4\geq 0$} which can be ensured if  we take $l_2\geq 1-\theta$,
and
{	\setlength{\abovedisplayskip}{10pt}
\setlength{\belowdisplayskip}{10pt}
\begin{align}\label{equ 455}
\mathbb{E}_{0}^{s+1}[\Psi^{s+1}_{1}]  \leq \Psi^{s}_{m} - \Gamma^{s}_m
\mathbb{E}_{0}^{s+1}[\|\vx^{s+1}_0 - \vx^{s+1}_{1}\|^2]
-\frac{5L^2}{\sigma^{A}_{min}\rho} \|\vx^s_{m}-\widetilde{\vx}^{s-1}\|^2.
\end{align}}
for any $s\in \{1,2,\cdots, S\}$ and $t\in\{1,2,\cdots,m\}$.
To establish the convergence of the sequence defined in equation (\ref{equ 21}), we calculate the above conditional expectations in expressions (\ref{equ 44}) and (\ref{equ 455}). By leveraging the property $\mathbb{E}[\mathbb{E}[\cdot \mid \mathcal{F}^{s}_{t}]]=\mathbb{E}[\cdot]$, we then evaluate the full expectation of (\ref{equ 44}) and (\ref{equ 455}). Additionally, we sum up the resulting expressions for (\ref{equ 44}) and (\ref{equ 455}) over the ranges $t=1,2,\ldots,m$ and $s=1,2,\ldots,S$ to obtain
{	
\begin{align}\label{equ 45}
\mathbb{E}\Psi_{m}^{S}- \mathbb{E}\Psi_{1}^{1}
\leq &-\gamma\sum_{s=1}^{S}\sum_{t=1}^{m}\mathbb{E}\|\vx_{t}^{s}-\vx_{t-1}^{s}\|^2  - \omega \sum_{s=1}^{S}\sum_{t=1}^{m} \mathbb{E}\|\vx_{t-1}^{s}-\widetilde{\vx}^{s-1}\|^2,
\end{align}}
where the parameter $\gamma = \min_{s,t} \Gamma^s_t$,
and
{	\setlength{\abovedisplayskip}{10pt}
\setlength{\belowdisplayskip}{10pt}
\begin{align}
\omega &= \min_{s,t}\{\left[h^{s+1}_{t+1}+\left(1+\alpha_1\right)(h^{s+1}_{t+1}+\beta_1)\right],\frac{5L^2}{\sigma^{A}_{min} \rho} \} \nonumber\\
&= \frac{5L^2}{\sigma^{A}_{min} \rho}.\nonumber
\end{align}}
From Assumption \ref{assum2}, there exists a low bound $\Psi^*$ of the sequence $\{\Psi^s_t\}$, i.e., $ \Psi^{s}_{t} \geq \Psi^*$.
Using the definition of $ R^{s}_t$, we have
{	\setlength{\abovedisplayskip}{10pt}
\setlength{\belowdisplayskip}{10pt}
\begin{align}
R^{\hat{s}}_{\hat{t}} = \min_{s,t} R^{s}_t \leq \frac{1}{\tau T} \mathbb{E}(\Psi^{1}_{1} - \Psi^*),
\end{align}}
where $\tau = \min(\gamma,\omega)$  and $T=mS$. This completes the whole proof.
\hfill $\blacksquare$

\section{Proof of the Property of $\left\|\vw_{t}^{s}-\vw_{t+1}^{s}\right\|^2$}\label{app:lemma 4}

In this section, we establish the linear convergence rate of our ASVRG-ADMM under the so-called
{KL} condition. We first draw the following Lemma of
the property of $\left\|\vw_{t+1}^{s}-\vw_{t}^{s}\right\|^2$,  where $\vw= \left(\vx, \vy, \vlambda \right)$ represents the sequence generated by our algorithm.

\begin{lemma}\label{lemma 4} {Let $\left\{\vw^{s}_{t}=\left(\vx_{t}^{s}, \vy_{t}^{s}, \vlambda_{t}^{s}\right)\right\}$ be the sequence generated by ADMM Algorithm \ref{alg1} under  Assumptions \ref{assum grad f}-\ref{Lip sub path}. Then}
{	
\begin{align}
\sum_{s=1}^{+\infty} \sum_{t=1}^{+\infty}\mathbb{E}\left\|\vw_{t+1}^{s}-\vw_{t}^{s}\right\|^{2}<+\infty. \nonumber
\end{align}}
\end{lemma}

\emph{Proof:} \quad  From the inequality (\ref{equ 45}), it follows that
{	
\begin{align}
&\tau \sum_{s=1}^{S} \sum_{t=1}^{m} R_{t}^{s} \leq { \mathbb{E}\Psi_{1}^{1}-\mathbb{E}\Psi_{m}^{s}} \leq +\infty , \nonumber\\
&\quad \sum_{s=1}^{S}\sum_{t=1}^{m} R_{t}^{s} \leq +\infty, \nonumber
\end{align}}
where $\tau = \min(\gamma,\omega)$ { with $\gamma$ and $\omega $   mentioned in the proof of Lemma \ref{thm 1}.}
By the definitions of $R_{t}^{s} $ in (\ref{equ 21}),  we can have that 
$$ \sum_{s=1}^{+\infty} \sum_{t=1}^{+\infty} \mathbb{E}\left\|\vx_{t+1}^{s}-\vx_{t}^{s}\right\|^{2}<+\infty$$.

It follows from  Lemma \ref{lemma of lambda} that $\vlambda$ can be bounded by $R_{t}^{s}$. Thus,
{	
\begin{align}
\sum_{s=1}^{+\infty} \sum_{t=1}^{+\infty} \mathbb{E}\left\|\vlambda_{t+1}^{s}-\vlambda_{t}^{s}\right\|^{2}<+\infty.\nonumber
\end{align}}

Next, we give an upper bound  of
$$ \sum_{s=1}^{+\infty} \sum_{t=1}^{+\infty} \mathbb{E}\left\|\vy_{t+1}^{s}-\vy_{t}^{s}\right\|^{2}$$
based on the following equations
\[
\left\{ \begin{array}{l}
\vlambda_{t+1}^{s}=\vlambda_{t}^{s}-\rho\left(A \vz_{t+1}^{s}+B \vy_{t+1}^{s}-c\right), \\
\vlambda_{t}^{s}=\vlambda_{t-1}^{s}-\rho\left(A \vz_{t}^{s}+B \vy_{t}^{s}-c\right)\\
\vlambda_{t+1}^{s}-\vlambda_{t}^{s}=\left(\vlambda _t^{s}-\vlambda_{t-1}^{s}\right)+\rho\left(A \vz_{t}^{s}-A \vz_{t+1}^{s}\right) +\rho\left(B \vy_{t}^{s}-B \vy_{t+1}^{s}\right).
\end{array}\right.
\]

Together with $\vx_{t+1}= \theta \vz_{t+1} + (1-\theta)\widetilde{\vx}$, simple algebra shows that
{	\setlength{\abovedisplayskip}{10pt}
\setlength{\belowdisplayskip}{10pt}
\begin{align}
&\mathbb{E}\left\|\rho\left(B \mathbf{y}_{t}^{s}-B \mathbf{y}_{t+1}^{s}\right)\right\|^{2}\nonumber\\
\leq & \mathbb{E}\left\|\left(\boldsymbol{\lambda}_{t+1}^{s}-\boldsymbol{\lambda}_{t}^{s}\right)-\left(\boldsymbol{\lambda}_{t}^{s}-\boldsymbol{\lambda}_{t-1}^{s}\right)-\rho\left(A \mathbf{z}_{t}^{s}-A \mathbf{z}_{t+1}^{s}\right) \right\|^{2} \nonumber\\
\leq & 3\mathbb{E}\left\|\boldsymbol{\lambda}_{t+1}^{s}-\boldsymbol{\lambda}_{t}^{s}\right\|^{2}+3 \mathbb{E}\left\|\boldsymbol{\lambda}_{t}^{s}-\boldsymbol{\lambda}_{t-1}^{s}\right\|^{2}+3 \rho^{2} \frac{\sigma_{\max }^{A}}{\theta^{2}} \mathbb{E}\left\|\mathbf{x}_{t+1}^{s}-\mathbf{x}_{t}^{s}\right\|^{2} \label{D.1}
\end{align}}

Then  we can use (\ref{D.1}) drive
{	\setlength{\abovedisplayskip}{10pt}
\setlength{\belowdisplayskip}{10pt}
\begin{align}
\mathbb{E}\left\|\vy_{t}^{s}-\vy_{t+1}^{s}\right\|^{2} &\leq \frac{3\bar{M}^2}{\rho^{2}} \mathbb{E}\left\|\vlambda_{t+1}^{s}-\vlambda_{t}^{s}\right\|^{2}+\frac{3\bar{M}^2}{\rho^{2}} \mathbb{E}\left\|\vlambda_{t}^{s}-\vlambda_{t-1}^{s}\right\|^{2}+\frac{3\bar{M}^2 \sigma_{\max }^{A}}{\theta^{2}} \mathbb{E}\left\|\vx_{t+1}^{s}-\vx_{t}^{s}\right\|^{2},\nonumber
\end{align}}
where the inequality is derived from  \cite[Lemma 1]{wang2019global} with the Assumption \ref{Lip sub path}.

We set $\zeta_{11} = \frac{3 \bar{M}^2}{\rho^2 }$ and $\zeta_{12} = \frac{3 \bar{M}^2 \sigma^A_{max}}{\theta^2}$, such that
{	\setlength{\abovedisplayskip}{10pt}
\setlength{\belowdisplayskip}{10pt}
\begin{align}\label{y-bound}
\sum_{s=1}^{+\infty} \sum_{t=1}^{+\infty} \mathbb{E}\left\|\vy_{t}^{s}-\vy_{t+1}^{s}\right\|^{2} &\leq \zeta_{11}\sum_{s=1}^{+\infty} \sum_{t=1}^{+\infty} \mathbb{E}\left\|\vlambda_{t+1}^{s}-\vlambda_{t}^{s}
\right\|^{2} + \zeta_{11}\sum_{s=1}^{+\infty} \sum_{t=1}^{+\infty} \mathbb{E}\left\|\vlambda_{t-1}^{s}-\vlambda_{t}^{s}\right\|^{2}\nonumber\\
&+ \zeta_{12}\sum_{s=1}^{+\infty} \sum_{t=1}^{+\infty} \mathbb{E}\left\|\vx_{t}^{s}-\vx_{t+1}^{s}\right\|^{2}.
\end{align}}
Recall that $R_{t}^{s}$ defined in (\ref{equ 21}) is
{	
\begin{align*}
R^{s}_t :=&\mathbb{E}\big[ \|\vx^{s}_{t}-\widetilde{\vx}^{s-1}\|^2 + \|\vx^{s}_{t-1}-\widetilde{\vx}^{s-1}\|^2+ \|\vx^{s}_{t+1}-\vx^{s}_t\|^2 + \|\vx^{s}_{t}-\vx^{s}_{t-1}\|^2\big].
\end{align*}}
By Lemma \ref{lemma of lambda}, the first term
{	
$$
\zeta_{11}\sum_{s=1}^{+\infty} \sum_{t=1}^{+\infty} \mathbb{E}\left\|\vlambda_{t+1}^{s}-\vlambda_{t}^{s}
\right\|^{2} + \zeta_{11}\sum_{s=1}^{+\infty} \sum_{t=1}^{+\infty} \mathbb{E}\left\|\vlambda_{t-1}^{s}-\vlambda_{t}^{s}\right\|^{2}
$$}
can be bounded as follows:
{	\setlength{\abovedisplayskip}{10pt}
\setlength{\belowdisplayskip}{10pt}
\begin{align}
&\zeta_{11}\sum_{s=1}^{+\infty} \sum_{t=1}^{+\infty} \mathbb{E}\left\|\vlambda_{t+1}^{s}-\vlambda_{t}^{s}
\right\|^{2} + \zeta_{11}\sum_{s=1}^{+\infty} \sum_{t=1}^{+\infty} \mathbb{E}\left\|\vlambda_{t-1}^{s}-\vlambda_{t}^{s}\right\|^{2}\nonumber\\
&\leq  \zeta_{13} R_{t}^{s},
\end{align}}
where $\zeta_{13}= 2 \zeta_{11} \text{max}\left\{\frac{5L^2}{\sigma^{A}_{min}}, \frac{5\phi_{\max}^2}{\sigma^{A}_{min}\eta^2}, \frac{5(\eta^2L^2+\phi_{\max}^2)}{\sigma^{A}_{min}\eta^2} \right\}$.
The second term $\zeta_{12}\sum_{s=1}^{+\infty} \sum_{t=1}^{+\infty} \mathbb{E}\left\|\vx_{t}^{s}-\vx_{t+1}^{s}\right\|^{2}$ in (\ref{y-bound}) can be bounded by
{	\setlength{\abovedisplayskip}{10pt}
\setlength{\belowdisplayskip}{10pt}
\begin{align}
\zeta_{12}\sum_{s=1}^{+\infty} \sum_{t=1}^{+\infty} \mathbb{E}\left\|\vx_{t}^{s}-\vx_{t+1}^{s}\right\|^{2} \leq \zeta_{12} R^{s}_t.
\end{align}}
Every term on the right-hand side of (\ref{y-bound}) can be bounded by $R_{t}^{s}$. Thus, there exists $\zeta_{14}>0 $ such that the upper bound of all above terms on the right-hand side can be limited by
{	
\begin{align}
\sum_{s=1}^{+\infty} \sum_{t=1}^{+\infty} \mathbb{E}\left\|\vy_{t}^{s}-\vy_{t+1}^{s}\right\|^{2} \leq \zeta_{14}\sum_{s=1}^{+\infty}\sum_{t=1}^{+\infty} R_{t}^{s} <+\infty,
\end{align}}
where $\zeta_{14} = \zeta_{12}+ \zeta_{13}$.
As a result, we obtain
{	\setlength{\abovedisplayskip}{10pt}
\setlength{\belowdisplayskip}{10pt}
\begin{align}
&\sum_{s=1}^{+\infty} \sum_{t=1}^{+\infty} \mathbb{E}\left\|\vy_{t}^{s}-\vy_{t+1}^{s}\right\|^{2}<+\infty,  \nonumber\\
& \sum_{s=1}^{+\infty} \sum_{t=1}^{+\infty} \mathbb{E}\left\|\vw_{t}^{s}-\vw_{t+1}^{s}\right\|^{2}
<+\infty. \nonumber
\end{align}}
This completes the whole proof.
\hfill $\blacksquare$

\setlength{\textfloatsep}{5pt}

\section{Proof of Lemma \ref{partial of Lagra}}\label{app:partial of Lagra}
Now, {based on Lemma \ref{lemma 4}  we} can demonstrate the upper bound of $\mathbb{E} \left\| \partial \mathcal {L}_\rho\left(\vw_{t+1}\right) \right\|$ which is important for the linear convergence {of ASVRG-ADMM}.

\begin{lemma} \label{partial of Lagra}
Let $\left\{\vw_t^{s}=\left(\vx_{t}^{s}, \vy_{t}^{s}, \vlambda_{t}^{s}\right)\right\}$ be the sequence generated by    Algorithm \ref{alg1} under  Assumptions \ref{assum grad f}-\ref{Lip sub path}. For notational simplicity, we omit the upper script $s$ with setting $\left(\vx_{j}, \vy_{j}, \vlambda_{j}\right):= \left(\vx_{t}^s, \vy_{t}^s, \vlambda_{t}^s \right)$, where $j= s*m +t$.
Then, there exists $\xi_{1}>0$ such that
{	\setlength{\abovedisplayskip}{10pt}
\setlength{\belowdisplayskip}{10pt}
\begin{align}
&\mathbb{E}_{t} \left\| \partial \mathcal {L}_\rho\left(\vw_{t+1}\right) \right\| \nonumber\\
\leq &\xi_{1} \big(\mathbb{E}_{t}[\left\|\vx_{t+1}-\widetilde{\vx}\right\|] +\left\|\vx_{t-1}-\widetilde{\vx}\right\| +\left\|\vx_{t}-\widetilde{\vx}\right\| +\left\|\vx_{t}-\vx_{t-1}\right\| +\mathbb{E}_{t}[\left\|\vx_{t+1}-\vx_{t}\right\|] \big).\nonumber
\end{align}}
\end{lemma}

Lemma \ref{partial of Lagra} shows that $\left\| \partial \mathcal {L}_\rho\left(\vw_{t+1}\right) \right\| $ in the Definition \ref{KŁ1} deducing the linear convergence with the  {KL} property, is upper bounded by some primal/iterative residuals. Based on this Lemma, we will show that the sequence $\{\vw_t^s\}$ converges to a critical point of the problem (\ref{equ1}).

\emph{Proof:} \quad From the definition of   $\mathcal {L}_\rho(.)$ in (\ref{equa2}), it follows
{	
\begin{align}
\frac{\partial \mathcal {L}_\rho\left(\vw_{t+1}\right)}{\partial \vx }=&\nabla f\left(\vx_{t+1}\right)-A^{\top} \vlambda_{t+1}+\rho A^{\top}\left(A \vx_{t+1}+B \vy_{t+1}-\vc\right), \nonumber\\
\frac{\partial\mathcal {L}_\rho\left(\vw_{t+1}\right)}{\partial \vy}=&\partial g\left(\vy_{t+1}\right)-B^{\top} \vlambda_{t+1}+ \rho B^{\top}\left(A \vx_{t+1}+B \vy_{t+1}-\vc\right), \nonumber\\
\frac{\partial \mathcal {L}_\rho\left(\vw_{t+1}\right)}{\partial \vlambda}=&-\left(A \vx_{t+1}+B \vy_{t+1}-\vc\right).\nonumber
\end{align}}

Recalling the first-order optimality conditions of the subproblems in Algorithm\ref{alg1} together with the update of $\vlambda_{t+1}$, we have
{	
\begin{align}\label{opt cond}
& \vx:\quad \nabla \hat{f}\left(\vx_{t}\right)=A^{\top} \vlambda_{t+1}-\frac{\theta}{\eta} Q\left(\vz_{t+1}-\vz_{t}\right), \nonumber\\
& \vy: \quad B^{\top} \vlambda_{t+1}-\rho B^{\top} A\left(\vx_{t}-\vz_{t+1}\right) \in \partial g\left(\vy_{t+1}\right),\nonumber\\
& \vlambda: \quad\vlambda_{t+1}=\vlambda_{t}- \rho\left(A \vz_{t+1}+B \vy_{t+1}-\vc\right).
\end{align}}
Invoking the above optimality conditions of the Algorithm\ref{alg1} yields
{	
\begin{align}\label{equ 49}
&
\frac{\vlambda_{t+1}-\vlambda_{t}}{\rho}+A \frac{1-\theta}{\theta}\left(\vx_{t+1}-\widetilde{\vx}\right)\in \partial_{\vlambda} \mathcal{L}_{\rho}\left(\vw_{t+1}\right) , \\
&B^{\top}\left(\vlambda_{t}-\vlambda_{t+1}\right)-\rho B^{\top} A \left( \vz_{t} - \vx_{t+1}\right) \in \partial_{\vy} \mathcal{L}_{\rho}\left(\vw_{t+1}\right), \\
&\nabla f\left(\vx_{t+1}\right)-\nabla \hat{f}\left(\vx_{t}\right)+\frac{1}{\eta} Q \left(\vx_{t}-\vx_{t+1}\right) +A^{\top}\left(\vlambda_{t}-\vlambda_{t+1}\right)\nonumber\\
&+\rho \frac{1-\theta}{\theta} A^{\top} A\left(\widetilde{\vx}-\vx_{t+1}\right) \in \partial_{\vx} \mathcal{L}_{\rho}\left(\vw_{t+1}\right).
\end{align}}
Thus,  we can obtain
{	
\begin{align}
& \mathbb{E}_{t} { dist\left(0, \partial \mathcal {L}_\rho\left(\vw_{t+1}\right)\right)}  \nonumber\\
= &  \mathbb{E}_{t}\left\|\partial_{\vlambda}\mathcal{L}_{\rho}\left(\vw_{t+1}\right)\right\|
+  \mathbb{E}_{t}\left\|\partial_{\vy} \mathcal{L}_{\rho}\left(\vw_{t+1}\right)\right\|
+ \mathbb{E}_{t}\left\|\partial_{\vx} \mathcal{L}_{\rho}\left(\vw_{t+1}\right)\right\| \nonumber\\
\leq &\mathbb{E}_{t}\Big[
\left\| \frac{\vlambda_{t+1}-\vlambda_{t}}{\rho}  \right\| +  \left\| A \frac{1-\theta}{\theta}\left(\vx_{t+1}-\widetilde{\vx}\right) \right\|  +  \left\| B^{\top}\left(\vlambda_{t}-\vlambda_{t+1}\right) \right\| +  \left\| \rho B^{\top} A \left( \vz_{t} - \vx_{t+1}\right) \right\| \nonumber\\
& + \left\| \nabla f\left(\vx_{t+1}\right)-\nabla \hat{f}\left(\vx_{t}\right)\right\|  +  \left\|\rho \frac{1-\theta}{\theta} A^{\top} A\left(\widetilde{\vx}-\vx_{t+1}\right)\right\|  +  \left\|\frac{1}{\eta} Q \left(\vx_{t}-\vx_{t+1}\right)\right\| + \left\|A^{\top}\left(\vlambda_{t}-\vlambda_{t+1}\right)\right\|\Big]
\nonumber\\
\leq &  \mathbb{E}_{t}\big[ \zeta_{21}\left\|\vlambda_{t+1}-\vlambda_{t}\right\| + \zeta_{22}\left\|\vx_{t+1}-\widetilde{\vx}\right\|+ \zeta_{23}\left\|\vx_{t+1}-\vx_{t}\right\| + \left\|\hat{\nabla}f(\vx_t)-\nabla f(\vx_{t+1}) \right\| \big]\nonumber \\
\mathop{\leq } &\zeta_{24}\mathbb{E}_{t}\Big( \left\|\vx_{t+1}-\widetilde{\vx}\right\| + \left\|\vx_{t}-\widetilde{\vx}\right\| + \left\|\vx_{t+1}-\vx_{t}\right\|  + \left\|\vlambda_{t+1}-\vlambda_{t}\right\| +  \left\|\hat{\nabla}f(\vx_t)-\nabla f(\vx_{t+1}) \right\| \Big)  \nonumber\\
\mathop{\leq }^{(i)}  
&
\zeta_{24}\Big( \mathbb{E}_{t}\left\|\vx_{t+1}-\widetilde{\vx}\right\| + \left\|\vx_{t}-\widetilde{\vx}\right\| + \mathbb{E}_{t}\left\|\vx_{t+1}-\vx_{t}\right\| +\mathbb{E}_{t}\left\|\hat{\nabla}f(\vx_{t+1})-\nabla f(\vx_{t}) \right\| \Big) \nonumber\\
&+\zeta_{24}\zeta_{25}\mathbb{E}_{t}\left\|\vx_{t-1}-\vx_{t} \right\| +\zeta_{24}\zeta_{25} \big(\left\|\vx_{t}-\widetilde{\vx}\right\| + \left\|\vx_{t-1}-\widetilde{\vx}\right\| + \mathbb{E}_{t}\left\|\vx_{t+1}-\vx_{t} \right\| \big),
\end{align}}
where $\sigma_{max}^{B}$ is the largest positive eigenvalue of $B^{\top} B$ (or equivalently the smallest positive eigenvalue of $B B^{\top}$ ), $\phi_{\max}$ is the largest positive eigenvalue of the matrix $Q$, and the inequality (i) is due to  Lemma \ref{lemma of lambda} and and the inequality
$\sqrt{a^{2}+b^{2}+c^{2}+d^{2}} \leq a+b+c+d, \text{for any}\quad a,b,c,d \geq 0$.

We can further obtain
{
\begin{align}
&\mathbb{E}_{t} \big[dist\left(0, \partial \mathcal {L}_\rho\left(\vw^{t+1}\right)\right)\big] \nonumber\\
\mathop{\leq }^{(i)}& \zeta_{24}L[\left\|\vx_{t}-\widetilde{\vx}\right\|] +\zeta_{24}L\mathbb{E}_{t}[\left\|\vx_{t+1}-\vx_{t} \right\|]
+\zeta_{24}\mathbb{E}_{t}\Big( \left\|\vx_{t+1}-\widetilde{\vx}\right\| + \left\|\vx_{t}-\widetilde{\vx}\right\| + \left\|\vx_{t+1}-\vx_{t}\right\|   \nonumber\\
\quad &+\zeta_{25}(\left\|\vx_{t}-\widetilde{\vx}\right\| + \left\|\vx_{t-1}-\widetilde{\vx}\right\| + \left\|\vx_{t+1}-\vx_{t} \right\| +\left\|\vx_{t-1}-\vx_{t} \right\|)  \Big)  \nonumber\\
\leq & (\zeta_{24}L+ \zeta_{24} +  \zeta_{24}\zeta_{25} )[\left\|\vx_{t}-\widetilde{\vx}\right\|]+ \zeta_{24}\mathbb{E}_{t}[\left\|\vx_{t+1}-\widetilde{\vx}\right\|] + \zeta_{24}\zeta_{25} [\left\|\vx_{t-1}-\widetilde{\vx}\right\|] \nonumber\\
\quad &+ (\zeta_{24}L +\zeta_{24} +\zeta_{24}\zeta_{25} )\mathbb{E}_{t}[\left\|\vx_{t+1}-\vx_{t}\right\|]+ \zeta_{24}\zeta_{25}[\left\|\vx_{t}-\vx_{t-1}\right\|], \nonumber\\
\mathop{\leq }& \xi_{1} \big( \mathbb{E}_{t}\left\|\vx_{t+1}-\widetilde{\vx}\right\|+\left\|\vx_{t-1}-\widetilde{\vx}\right\|+ \mathbb{E}_{t}\left\|\vx_{t+1}-\vx_{t}\right\| +\left\|\vx_{t}-\widetilde{\vx}\right\|+\left\|\vx_{t}-\vx_{t-1}\right\|\big),\nonumber
\end{align}
where the parameters $\zeta_{21}, \zeta_{22}$ and $\zeta_{23}$ are
\begin{align} \label{equ def zeta}
\left\{\begin{array}{l}
\zeta_{21}=\frac{1}{\rho}+\sqrt{\sigma_{\max }^{B}}+\sqrt{\sigma_{\max }^{A}}, \\
\zeta_{22}=\rho \frac{1-\theta}{\theta} \sigma_{\max }^{A}, \\
\zeta_{23}=\frac{\phi_{\max }}{\eta}+ \rho \sqrt{\sigma_{\max }^{B} \sigma_{\max }^{A}} \\
\zeta_{24}=\max \left\{\zeta_{21}, \zeta_{22}, \zeta_{23}, 1\right\}, \\
\zeta_{25}=\max \left\{\sqrt{\frac{5 L^{2}}{\sigma_{\min }^{A}}}, \sqrt{\frac{5 \phi_{\max }^{2}}{\sigma_{\min }^{A} \eta^{2}}}, \sqrt{\frac{5\left(\eta^{2} L^{2}+\phi_{\max }^{2}\right)}{\sigma_{\min }^{A} \eta^{2}}}\right\}, \\
\xi_{1}\ =\zeta_{24} L+\zeta_{24}+\zeta_{24} \zeta_{25},
\end{array}\right.
\end{align}}
where the inequality (i) is due to the triangle inequality
$\| a+b \| \leq  \|a\|+\|b\|$, Assumption \ref{assum grad f}, the inequality $(\mathbb{E}\|\vx\|)^{2} \leq \mathbb{E}(\|\vx\|^{2})$ and the inequality (\ref{bjc deta}).

This completes the  proof.
\hfill $\blacksquare$

\section{Proof of Lemma \ref{thm 2}}\label{app:thm 2}
The convergence properties of the stochastic sequence $\left\{\vw_{t}=\left(\vx_{t}, \vy_{t}, \vlambda_{t}\right)\right\}$ under the Kurdyka-Lojasiewicz (KL) inequality condition will be investigated in the following sections.
It is important to acknowledge that the implementation of the KL technique displays slight differences between stochastic and deterministic algorithms.
For further details, readers are encouraged to refer to \cite{milzarek2023convergence,chouzenoux2023kurdyka}.			

Before proving the key Lemma \ref{thm 2}, we first prove the Lemma \ref{lemma 6}.

\begin{lemma}\label{lemma 6}
Let $\left\{\vw_{t}=\left(\vx_{t}, \vy_{t}, \vlambda_{t}\right)\right\}$ (for notational simplicity, we omit the label s) be the stochastic sequence generated by ADMM procedure. Let $S\left(\vw_{0}\right)$ denote the set of its limit points. With Definition \ref{kkt}, then we have

\begin{itemize}

\item[i)]$S\left(\vw_{0}\right)$ is a.s. a nonempty compact set, and $ dist\left(\vw_{t}, S\left(\vw_{0}\right)\right) \text{converges a.s. to 0};$

\item[ii)] $S\left(\vw_{0}\right) \subset {  \operatorname{crit} \mathcal{L}_{\rho}}$ a.s.;

\item[iii)] $\mathcal{L}_{\rho}(\cdot)$ is a.s. finite and constant on $S\left(\vw_{0}\right)$, equal to $\inf _{t \in N} \mathcal{L}_{\rho}\left(\vw_{t}\right)=\lim _{t \rightarrow+\infty} \mathcal{L}_{\rho}\left(\vw_{t}\right)$ a.s.

\end{itemize}
\end{lemma}


\emph{Proof:} \quad	We prove the results item by item.

i) 
 By applying the descent inequalities (\ref{equ 44}) and (\ref{equ 455}) along with the supermartingale convergence theorem, we can establish that 
\begin{align}\label{a.s. w}
\left\{\begin{array}{l}\sum_{s=1}^{+\infty} \sum_{t=1}^{+\infty}\left\|\mathbf{x}_{t+1}^s-\mathbf{x}_t^s\right\|^2<+\infty \quad \text { a.s.},\\ \sum_{s=1}^{+\infty} \sum_{t=1}^{+\infty}\left\|\mathbf{y}_{t+1}^s-\mathbf{y}_t^s\right\|^2<+\infty \quad \text { a.s.},\\ \sum_{s=1}^{+\infty} \sum_{t=1}^{+\infty}\left\|\vlambda_{t+1}^s-\vlambda_t^s\right\|^2<+\infty \quad \text { a.s.},\\
\sum_{s=1}^{+\infty} \sum_{t=1}^{+\infty}\left\|\vw_{t+1}^s-\vw_t^s\right\|^2<+\infty \quad \text { a.s.},	
\end{array}\right.
\end{align}
then we can further have $\left\|\vw_{t+1}-\vw_{t}\right\| \rightarrow 0 \quad \text { a.s.}$

Consequently,  for any sequence satisfying $\left\|\vw_{t+1}-\vw_{t}\right\| \rightarrow 0 \quad \text { a.s.}$, claim i) holds.
And we refer to \cite[propisition 2.3]{chouzenoux2023kurdyka}, \cite[Lemma A.14]{BLZhang21} for more details.

ii) Let $\left(\vx^{*}, \vy^{*}, \vlambda^{*}\right) \in S\left(\vw_{0}\right)$,  then there exists a subsequence $\left\{\left(\vx_{t_{j}}, \vy_{t_{j}}, \vlambda_{t_{j}}\right)\right\}$ of $\left\{\left(\vx_{t}, \vy_{t}, \vlambda_{t}\right)\right\}$ converging a.s. to $\left(\vx^{*}, \vy^{*}, \vlambda^{*}\right)$. Note that (\ref{a.s. w}) implies	
\begin{align} \label{equ 50}
\left\|\vw_{t+1}-\vw_{t}\right\| \rightarrow 0\quad \text{a.s},
\end{align}
which means $\vw_{t+1}$ converges a.s. to $\vw_{t}$ and $\left\{\left(\vx_{t_{j}+1}, \vy_{t_{j}+1}, \vlambda_{t_{j}+1}\right)\right\}$ also converges a.s. to $\left(\vx^{*}, \vy^{*}, \vlambda^{*}\right)$.
It is important to highlight that the convergence of random variables in an almost sure manner is observed in $\vw_{t+1}  \rightarrow \vw_{t}\quad \text{a.s}$. As a result, any newly derived conclusions based on this convergence will also hold with an almost sure guarantee.  As a result, many of the conclusions presented in the subsequent chapters differ from those in the previous chapters but can be considered almost surely valid.
Since $\vy_{t_{j}+1}$ is a minimizer of $\mathcal{L}_{\rho}\left(\vx_{t_{j}}, \vy, \vlambda_{t_{j}}\right)$ for the variable $\vy$, it holds that
{	
\begin{align}\label{equ 51}
\mathcal{L}_{\rho}\left(\vx_{t_{j}}, \vy_{t_{j}+1}, \vlambda_{t_{j}}\right) \leq \mathcal{L}_{\rho}\left(\vx_{t_{j}}, \vy^{*}, \vlambda_{t_{j}}\right).
\end{align}}

Then, it follows from Equations $(\ref{equ 50}), (\ref{equ 51})$ and the continuity of $\mathcal{L}_{\rho}(\cdot)$ with respect to $\vx$ and $\vlambda$ that
{	
\begin{align}\label{equ 52}
&\limsup _{j \rightarrow+\infty} \mathcal{L}_{\rho}\left(\vx_{t_{j}}, \vy_{t_{j}+1}, \vlambda_{t_{j}}\right) \nonumber\\
=&\limsup _{j \rightarrow+\infty} \mathcal{L}_{\rho}\left(\vx_{t_{j}+1}, \vy_{t_{j}+1}, \vlambda_{t_{j}+1}\right)\leq \mathcal{L}_{\rho}\left(\vx^{*}, \vy^{*}, \vlambda^{*}\right) a.s.
\end{align}}
On the other hand, by the lower semicontinuity of $\mathcal{L}_{\rho}(\cdot)$, we know
{	
\begin{align}\label{equ 53}
\liminf _{j \rightarrow+\infty} \mathcal{L}_{\rho}\left(\vx_{t_{j}+1}, \vy_{t_{j}+1}, \vlambda_{t_{j}+1}\right) \geq \mathcal{L}_{\rho}\left(\vx^{*}, \vy^{*}, \vlambda^{*}\right) a.s.
\end{align}}
The above two relations (\ref{equ 52}) and (\ref{equ 53}) show that $\lim _{j \rightarrow+\infty} g\left(\vy_{t_{j}+1}\right)=g\left(\vy^{*}\right)$ a.s. Because of the continuity of $\nabla f$ and the closeness of $\partial g$, taking limits in {  Equation (\ref{opt cond})} along the subsequence $\left\{\left(\vx_{t_{j}+1}, \vy_{t_{j}+1}, \vlambda_{t_{j}+1}\right)\right\}$ and using Equation (\ref{equ 50}) again, we almost surely have
$$
\begin{aligned}
&A^{\mathrm{T}} \vlambda^{*} = \partial f\left(\vx^{*}\right), \\
& B^{T}\vlambda^{*}\in \nabla g\left(\vy^{*}\right),  \\
&A \vx^{*}+B\vy^{*}-\vc=0.
\end{aligned}
$$
Then, $\left(\vx^{*}, \vy^{*}, \vlambda^{*}\right)$ is a.s. a critical point of the problem $(\ref{equ1})$, hence $w^{*} \in$ crit $\mathcal{L}_{\beta}$ a.s.

iii) For any point $\left(\vx^{*}, \vy^{*}, \vlambda^{*}\right) \in S\left(w^{0}\right)$, there exists a subsequence $\left\{\left(\vx_{t_{j}}, \vy_{t_{j}}, \vlambda_{t_{j}}\right)\right\}$ of $\left\{\left(\vx_{t}, \vy_{t}, \vlambda_{t}\right)\right\}$ converging to $\left(\vx^{*}, \vy^{*}, \vlambda^{*}\right)$. Combining Equations $(\ref{equ 52})$ and $(\ref{equ 53})$, we obtain
$$
\lim _{j \rightarrow+\infty} \mathcal{L}_{\rho}\left(\vx_{t_{j}}, \vy_{t_{j}}, \vlambda_{t_{j}}\right)=\mathcal{L}_{\rho}\left(\vx^{*}, \vy^{*}, \vlambda^{*}\right) \quad \text{a.s.}
$$

Therefore, $\mathcal{L}_{\rho}(\cdot)$ is a.s. constant on $S\left(\vw^{0}\right)$. Moreover, $\inf _{t \in N} \mathcal{L}_{\rho}\left(\vw_{t}\right)=\lim _{t \rightarrow+\infty} \mathcal{L}_{\rho}\left(\vw_{t}\right)$ a.s.
\hfill $\blacksquare$

~\\

With the established conclusions, we are now prepared to give the proof of Lemma \ref{thm 2}.

\emph{Proof:} \quad  From the proof of Lemma \ref{lemma 6}, it follows that $\mathcal{L}_{\rho}\left(\vw_{t}\right) \xrightarrow{a.s.} \mathcal{L}_{\rho}\left(\vw^{*}\right)$ for all $\vw^{*} \in S\left(\vw^{0}\right)$ . We consider two cases.

(i)We first consider this case: there exists an ﬁnite positive discrete random variable $t_{0}$ for which $\Psi_{t_{0}}=\mathcal{L}_{\rho}\left(\vw^{*}\right)$ a.s.

Taking full expectation operator to the inequality (\ref{equ 44}), we can have 
\begin{align}
&\mathbb{E}\Psi_{t+1} \mathop{\leq }^{(i)} \mathbb{E}\Psi_{t}-\Gamma_{t}\mathbb{E}\left\|\vx_{t+1}-\vx_{t}\right\|^{2}- w\mathbb{E}\left\|\vx_{t-1}-\widetilde{\vx}\right\|^{2} ,\nonumber\\
&\Gamma_{t} \mathbb{E}\left\|\vx_{t+1}-\vx_{t}\right\|^{2} + w\mathbb{E}\left\|\vx_{t-1}-\widetilde{\vx}\right\|^{2} 
\leqslant  \mathbb{E}\Psi_{t}-\mathbb{E}\Psi_{t+1} \leq \mathbb{E}\Psi_{t_{0}}-\mathbb{E}\mathcal{L}_{\rho}\left(\vw^{*}\right)=0\quad \text{a.s.}, 
\end{align}}
where  the parameters $\Gamma_{t}$, $w$ are defined in (\ref{equ 45}).
Then we can drive
{	
\begin{align}\label{xt x}
\Gamma_{t}\mathbb{E}\left\|\vx_{t+1}-\vx_{t}\right\|^{2} + w\mathbb{E}\left\|\vx_{t-1}-\widetilde{\vx}\right\|^{2} &= 0, \nonumber\\
\vx_{t+1}=\vx_{t} , \vx_{t-1}&=\widetilde{\vx} \quad \text{a.s.}, \nonumber\\
\vlambda_{t+1}=\vlambda_{t},  \vy_{t+1} &=\vy_{t} \quad \text{a.s.}  \tag{S.48}
\end{align}

Thus, for any $t \geq t_{0}$, we have
$\vx_{t+1}=\vx_{t}$ ,  $\vx_{t-1}=\widetilde{\vx}$, $\vlambda_{t+1}=\vlambda_{t}$ and $\vy_{t+1}=\vy_{t}  \quad \text{a.s.}$ Then the assertion holds.

(ii) We define  $\Pi=\liminf _{t \rightarrow+\infty}\left\{\omega \in \Omega \mid \Psi_{t}=\Psi_{t}(\omega)>\mathcal{L}_{\rho}\left(\vw^{*}(\omega)\right)\right\}$. We assume $\mathbb{P}(\Pi)=1$ and
$\Psi_{t}>\mathcal{L}_{\rho}\left(\vw^{*}\right)$ 
for all $t$ over the set $\Pi$.
Since $dist\left(\vw_{t}, S\left(\vw^{0}\right)\right) \xrightarrow{a.s.} 0$, it follows that for all $\epsilon>0$, there exists a ﬁnite positive discrete random variable $k_{1}>0$, such that for any $t>k_{1}, dist\left(\vw_{t}, S\left(\vw^{0}\right)\right)<\epsilon$ a.s. Again since $\Psi_{t} \xrightarrow{a.s.} \mathcal{L}_{\rho}\left(\vw^{*}\right)$, it follows for all $\eta>0$, there exists a ﬁnite positive discrete random variable $k_{2}>0$, such that for any $t>k_{2}, \Psi_{t} <\mathcal{L}_{\rho}\left(\vw^{*}\right)+\eta$ a.s. Consequently, for all $\epsilon, \eta>0$, when $t>\tilde{k}=\max \left\{k_{1},  k_{2}\right\}$, we almost surely have
$dist\left(\vw_{t}, S\left(\vw^{0}\right)\right)<\epsilon$, $\mathcal{L}_{\rho}\left(\vw^{*}\right)<\Psi_{t}<\mathcal{L}_{\rho}\left(\vw^{*}\right)+\eta. $
Since $S\left(\vw^{0}\right)$ is a.s. a nonempty compact set and $\mathcal{L}_{\rho}(\cdot)$ is a.s. constant on $S\left(\vw^{0}\right)$, applying the definition of KL property with $\Omega=S\left(\vw^{0}\right)$, we deduce that for any $t>\tilde{k}$
$$
\nabla \varphi\left(\Psi_{t}-\mathcal{L}_{\rho}\left(\vw^{*}\right)\right) dist\left(0, \partial \Psi_{t}\right) \geq 1 \quad \text{a.s.}
$$

We denote the $\Psi^{*} = \mathcal{L}_{\rho}\left(\vw^{*}\right)$, then the above inequality becomes
{	
\begin{align}\label{psi and d}
\nabla \varphi \left(\Psi_{t}-\Psi^{*}\right) dist\left(0, \partial \Psi_{t}\right) \geq 1 \quad \text{a.s.} \tag{S.49}
\end{align}}

From the concavity of $\varphi$, we get that
{	
\begin{align}\label{equ 56}
\varphi\left(\Psi_{t}-\Psi^{*}\right)-\varphi\left(\mathbb{E}_{t}\Psi_{t+1}-\Psi^{*}\right)
\geqslant & \nabla\varphi\left(\Psi_{t}-\Psi^{*}\right)\left(\Psi_{t}-\mathbb{E}_{t}\Psi_{t+1} \right) \quad\text{a.s.},  \tag{S.50} 
\end{align}}
and it amount to
{	
\begin{equation}
\begin{aligned}
\Psi_{t}-\mathbb{E}_{t}\Psi_{t+1} & \leqslant \frac{\varphi\left(\Psi_{t}-\Psi^{*}\right)-\varphi\left(\mathbb{E}_{t}\Psi_{t+1}-\Psi^{*}\right)}{\nabla \varphi\left(\Psi_{t}-\Psi^{*}\right)} \\
& \stackrel{(i)}{\leq}\left[\varphi\left(\Psi_{t}-\Psi^{*}\right)-\varphi\left(\mathbb{E}_{t}\Psi_{t+1}-\Psi^{*}\right)\right] \times dist\left(0, \partial \Psi_{t}\right) \quad \text{a.s.},
\end{aligned}\nonumber
\end{equation}}
where the inequality (i) holds by the inequality (\ref{psi and d}), and we set
{	 \begin{align}\label{def of delta}
\Pi_{p, q}:=\varphi\left(\Psi_{p}-\Psi^{*}\right)-\varphi\left(\Psi_{q}-\Psi^{*}\right). \tag{S.51}
\end{align}}
Moreover, recalling the equations (\ref{equ 44}) and (\ref{equ 455}), we have
{	\setlength{\abovedisplayskip}{10pt}
\setlength{\belowdisplayskip}{10pt}
\begin{align}\label{equ 57}
\Psi_{t}-\mathbb{E}_{t}\Psi_{t+1} \geqslant \tau \left(\mathbb{E}_{t}\left\|\vx_{t+1}-\vx_{t}\right\|^{2}+\left\|\vx_{t-1}-\widetilde{\vx}\right\|^{2}\right),
\tag{S.52}
\end{align}}
where $\tau = \min(\gamma,\omega)$ with $\gamma$ and $\omega$ being given by Theorem 1.

From the definition of the sequence $\{(\Psi^{s}_{t})_{t=1}^m\}_{s=1}^S$, we obtain
{	\setlength{\abovedisplayskip}{10pt}
\setlength{\belowdisplayskip}{10pt}
\begin{align}\label{equ 58}
dist\left(0, \partial \Psi_{t}\right) &=dist\big(0,( \partial\mathcal{L}_{\rho}\left(\vw_{t}\right)+2 h_{t}\left(\vx_{t}-\widetilde{\vx}\right)+2 \beta_5\left(\vx_{t}-\vx_{t-1}\right)) \big)\nonumber\\
& \leq dist\left(0, \partial\mathcal{L}_{\rho}\left(\vw_{t}\right)\right)+2 h_{t}\left\|\vx_{t}-\widetilde{\vx}\right\|+2 \beta_5 \| \vx_{t}-\vx_{t-1} \|.
\tag{S.53}
\end{align}}
Combine the (\ref{equ 56})-(\ref{equ 58}) we have
{	\setlength{\abovedisplayskip}{10pt}
\setlength{\belowdisplayskip}{10pt}
\begin{align}
&\tau \left(\mathbb{E}_{t}\left\|\vx_{t+1}-\vx_{t}\right\|^{2}+\left\|\vx_{t-1}-\widetilde{\vx}\right\|^{2}\right) \nonumber\\
\leq &\big[
\varphi\left(\Psi_{t}-\Psi^{*}\right)-\varphi\left(\mathbb{E}_{t}\Psi_{t+1}-\Psi^{*}\right)\big]
\bigg[dist\left(0, \partial\mathcal{L}_{\rho}\left(\vw_{t}\right)\right)+2 h_{t}\left\|\vx_{t}-\widetilde{\vx}\right\|+2 \beta_5\| \vx_{t}-\vx_{t-1} \|\bigg]
\quad\text{a.s.}\tag{S.54}
 \label{e12}
\end{align}}
Taking the full expectation and inserting Lemma \ref{partial of Lagra} into the above inequality, we see that
\begin{align}\label{equ 701}
&\mathbb{E}[dist\left(0, \partial \Psi_{t}\right)] +2 h_{t}\mathbb{E}\left\|\vx_{t}-\widetilde{\vx}\right\|
+2 \beta_5 \mathbb{E}\left\| \vx_{t}-\vx_{t-1}\right\| \nonumber\\
\leq & \xi_{max} \bigg[\mathbb{E}\left\|\vx_{t+1}-\widetilde{\vx}\right\|+\mathbb{E}\left\|\vx_{t-1}-\widetilde{\vx}\right\| +\mathbb{E}\left\|\vx_{t+1}-\vx_{t}\right\| +\mathbb{E}\left\|\vx_{t}-\widetilde{\vx}\right\|+\mathbb{E}\left\|\vx_{t}-\vx_{t-1}\right\| \bigg],
\tag{S.55}
\end{align}
where $\xi_{max} = \xi_1 + 2h_t +2\beta_5$ and the inequality (\ref{equ 701}) becomes
\begin{align}
\mathbb{E}[dist\left(0, \partial \Psi_{t}\right)]\leq  &\xi_{max} \bigg[\mathbb{E}\left\|\vx_{t+1}-\widetilde{\vx}\right\|+\mathbb{E}\left\|\vx_{t-1}-\widetilde{\vx}\right\| +\mathbb{E}\left\|\vx_{t+1}-\vx_{t}\right\|\nonumber\\ &+\mathbb{E}\left\|\vx_{t}-\widetilde{\vx}\right\|+\mathbb{E}\left\|\vx_{t}-\vx_{t-1}\right\| \bigg].\tag{S.56}
 \label{e14}
\end{align}

Now, combine (\ref{e12}) and (\ref{e14}) and take the full expectation on (\ref{e12}) to obtain
\begin{align}
&\tau   \left(\mathbb{E}\left\|\vx_{t+1}-\vx_{t}\right\|^{2}+\mathbb{E}\left\|\vx_{t-1}-\widetilde{\vx}\right\|^{2}\right) \nonumber\\
\leq & \xi_{max} 
\mathbb{E}\big[
\varphi\left(\Psi_{t}-\Psi^{*}\right)-\varphi\left(\mathbb{E}_{t}\Psi_{t+1}-\Psi^{*}\right)\big]
\mathbb{E} \bigg[\left\|\vx_{t+1}-\widetilde{\vx}\right\|+\left\|\vx_{t-1}-\widetilde{\vx}\right\|  + \left\|\vx_{t+1}-\vx_{t}\right\|\nonumber\\ &+\left\|\vx_{t}-\tilde {\vx}\right\|+\left\|\vx_{t}-\vx_{t-1}\right\| \bigg]
\quad \text{a.s.}\nonumber
\end{align}

It follows from the inequality $\sqrt{a^{2}+b^{2}} \geqslant \frac{\sqrt{2}}{2}(a+b) $ that
{	
\begin{align}
&4\left(\mathbb{E}\left\|\vx_{t+1}-\vx_{t}\right\|+\mathbb{E}\left\|\vx_{t-1}-\widetilde{\vx}\right\|\right)\nonumber\\
\leq &4 \sqrt{\xi_{max}} \sqrt{\frac{2}{\tau}} \bigg[\mathbb{E}\left\|\vx_{t+1}-\widetilde{\vx}\right\| +\mathbb{E}\left\|\vx_{t-1}-\widetilde{\vx}\right\|+\mathbb{E}\left\|\vx_{t+1}-\vx_{t}\right\| +\mathbb{E}\left\|\vx_{t}-\widetilde{\vx}\right\|+\mathbb{E}\left\|\vx_{t}-\vx_{t-1}\right\| \bigg]^{\frac{1}{2}}\nonumber\\
&\times \left\{\mathbb{E}\big[
\varphi\left(\Psi_{t}-\Psi^{*}\right)-\varphi\left(\mathbb{E}_{t}\Psi_{t+1}-\Psi^{*}\right)\big]\right\} ^{\frac{1}{2}} \quad \text{a.s.},
\tag{S.57}
\end{align}}
which by the Cauchy-Schwartz inequality $2 \sqrt{ab} \leq a+b$  further gives
\begin{align}\label{equ 61}
&4\left(\mathbb{E}\left\|\vx_{t+1}-\vx_{t}\right\|+\mathbb{E}\left\|\vx_{t-1}-\widetilde{\vx}\right\|\right)\nonumber\\ \leq & \Big(2 \sqrt{\xi_{max}} \sqrt{\frac{2}{\tau}} \Big)^2 \mathbb{E}\big[
\varphi\left(\Psi_{t}-\Psi^{*}\right)-\varphi\left(\mathbb{E}_{t}\Psi_{t+1}-\Psi^{*}\right)\big] +\bigg[\mathbb{E}\left\|\vx_{t+1}-\widetilde{\vx}\right\|+\mathbb{E}\left\|\vx_{t-1}-\widetilde{\vx}\right\|   \nonumber\\
&+\mathbb{E}\left\|\vx_{t+1}-\vx_{t}\right\|+\mathbb{E}\left\|\vx_{t}-\widetilde{\vx}\right\|+\mathbb{E}\left\|\vx_{t}-\vx_{t-1}\right\| \bigg],\nonumber\\
\stackrel{(i)}{\leq} &\Big(2 \sqrt{\xi_{max}} \sqrt{\frac{2}{\tau}} \Big)^2  \mathbb{E}\big[
\varphi\left(\Psi_{t}-\Psi^{*}\right)-\varphi\left(\Psi_{t+1}-\Psi^{*}\right)\big] +\bigg[\mathbb{E}\left\|\vx_{t+1}-\widetilde{\vx}\right\|+\mathbb{E}\left\|\vx_{t-1}-\widetilde{\vx}\right\|  \nonumber\\
&+\mathbb{E}\left\|\vx_{t+1}-\vx_{t}\right\| +\mathbb{E}\left\|\vx_{t}-\widetilde{\vx}\right\|+\mathbb{E}\left\|\vx_{t}-\vx_{t-1}\right\| \bigg],\nonumber\\
=& \Big(2 \sqrt{\xi_{max}} \sqrt{\frac{2}{\tau}} \Big)^2  \mathbb{E}(\Pi_{t, t+1}) +\bigg[\mathbb{E}\left\|\vx_{t+1}-\widetilde{\vx}\right\|+\mathbb{E}\left\|\vx_{t-1}-\widetilde{\vx}\right\| +\mathbb{E}\left\|\vx_{t+1}-\vx_{t}\right\|  +\mathbb{E}\left\|\vx_{t}-\widetilde{\vx}\right\|\nonumber\\
&+\mathbb{E}\left\|\vx_{t}-\vx_{t-1}\right\| \bigg] \quad \text{a.s.}
\tag{S.58}
\end{align}
In the above analysis, we use the abbreviation of $\mathbb{E}[\cdot \mid \mathcal{F}_{t}]$ as $\mathbb{E}_{t}[\cdot ]$, and  apply the conditional Jensen’s inequality to concave function $\varphi$ as
\begin{align}
&\mathbb{E}\big[
\varphi\left(\mathbb{E}_{t}\Psi_{t+1}-\Psi^{*}\right)\big] = \mathbb{E}\big[
\varphi\left(\mathbb{E}\Psi_{t+1}-\Psi^{*}\mid \mathcal{F}_{t} \right)\big]\nonumber\\ \geq 
&\mathbb{E}\big[\mathbb{E}\big[
\varphi\left(\Psi_{t+1}-\Psi^{*}\right)\mid \mathcal{F}_{t}\big]\big]=
\mathbb{E}\big[
\varphi\left(\Psi_{t+1}-\Psi^{*}\right)\big].\nonumber
\end{align}

Summing up  (\ref{equ 61}) from
$t =\tilde{t}+ 1,\cdots, m$  yields
{ 
\begin{align}\label{equ 62}
&\quad 2\mathbb{E}\left\|\vx_{\tilde{t}}-\widetilde{\vx}\right\|+\mathbb{E}\left\|\vx_{\tilde{t}+1}-\widetilde{\vx}\right\|+  2\sum_{t=\tilde{t}+1}^{m}\mathbb{E}\left\|\vx_{t+1}-\vx_{t}\right\|+\sum_{t=\tilde{t}+1}^{m}\mathbb{E}\left\|\vx_{t-1}-\widetilde{\vx}\right\|  +\mathbb{E}\left\|\vx_{m+1}-\vx_{m}\right\| \nonumber\\
& \leqslant  2\mathbb{E}\left\|\vx_{m}-\widetilde{\vx}\right\|+\mathbb{E}\left\|\vx_{m+1}-\widetilde{\vx}\right\| + (2 \sqrt{\xi_{max}} \sqrt{\frac{2}{\tau}} )^2\times \mathbb{E}[\varphi\left(\psi_{\tilde{t}+1}-\psi^{*} \right)-\varphi\left(\psi_{m+1}-\psi^{*}\right)] \nonumber\\
&\quad +\mathbb{E}\left\|\vx_{\tilde{t}+1}-\vx_{\tilde{t}}\right\|\nonumber\\
&< + \infty \quad \text{a.s.}
\tag{S.59}
\end{align}}
Let $m$  be $+\infty$, by the equation (\ref{xt x}), it yields
{ 
\[
\sum_{t=\tilde{t}+1}^{+\infty}\mathbb{E}\left\|\vx_{t+1}-\vx_{t}\right\|
<+\infty \quad \text{a.s.} \quad\textrm{and}\quad
\sum_{t=\tilde{t}+1}^{+\infty}\mathbb{E}\left\|\vx_{t+1}-
\widetilde{\vx}\right\|<+\infty \quad \text{a.s.}
\]}
Similarly we can obtain that
{	
\[
\sum_{t=\tilde{t}+1}^{+\infty}\mathbb{E}\left\|\vy_{t+1}-\vy_{t}
\right\|<+\infty \quad \text{a.s.} \quad\textrm{and}\quad
\sum_{t=\tilde{t}+1}^{+\infty}\mathbb{E}\left\|
\vlambda_{t+1}-\vlambda_{t}\right\|
<+\infty \quad \text{a.s.}
\]}
Hence, $\sum_{k=1}^{+\infty}\mathbb{E}\left\|\vx_{t+1}-\vx_{t}\right\|<+\infty \quad \text{a.s.}$ Besides, we note that
{	
$$
\begin{aligned}
&\mathbb{E}\left\|\vw_{t+1}-\vw_{t}\right\| \\ =&\mathbb{E}\left(\left\|\vx_{t+1}-\vx_{t}\right\|^{2}+\left\|\vy_{t+1}-\vy_{t}\right\|^{2}+\left\|\vlambda_{t+1}-\vlambda_{t}\right\|^{2}\right)^{1 / 2} \\
\leq & \mathbb{E}\left\|\vx_{t+1}-\vx_{t}\right\|+\mathbb{E}\left\|\vy_{t+1}-\vy_{t}\right\|+\mathbb{E}\left\|\vlambda_{t+1}-\vlambda_{t}\right\| \quad \text{a.s.}
\end{aligned}
$$}
Consequently, 
\begin{align}
&\sum_{t=0}^{+\infty}\mathbb{E}\left\|\vw_{t+1}-\vw_{t}\right\|<+\infty \quad \text{a.s.},\nonumber \\ &\sum_{t=0}^{+\infty}\left\|\vw_{t+1}-\vw_{t}\right\|<+\infty  \quad \text{a.s.}
\tag{S.60}
\end{align}
This completes the whole proof.
\hfill $\blacksquare$

\begin{remark}
In Lemma \ref{thm 1}, we have presented the framework of convergence analysis and obtained
$$
\sum_{t=0}^{+\infty}\left\|\vw_{t+1}-\vw_{t}\right\|^{2} <+\infty \quad \text{a.s.}
$$
However, it is unclear whether their results can be extended to
$$
\sum_{t=0}^{+\infty}\left\|\vw_{t+1}-\vw_{t}\right\|<+\infty \quad \text{a.s.}
$$ 

Lemma \ref{thm 2} gives some sufficient conditions to guarantee the finite length property. Combining Lemma \ref{lemma 4} and Lemma \ref{thm 2}, we can draw the gradient can be bounded by the iteration points.
\end{remark}

\section{Proof of Theorem \ref{thm 3}}\label{app:thm 3}
Now, under the   {KL} property defined in Definition \ref{KŁ1} (see \cite{attouch2010proximal}), we make full use of the decreasing property of the potential energy function $\Psi_t^s$ and the boundedness of $dist\left(0, \partial \Psi_{t}\right)$ (see Lemma \ref{partial of Lagra}) to prove Theorem \ref{thm 3}.

\emph{Proof:} \quad 	 First, we prove a newly defined sequence $N_t$ in (\ref{Nt}) converges a.s. to zero Q-linearly.
By the KL property at $(\vx^{*}, \vy^{*}, \vlambda^{*})$ we have
\begin{align} \label{grad psi}
\nabla \varphi \left(\Psi_{\tilde{t}+1}-\Psi^{*}\right) dist\left(0, \partial \Psi_{\tilde{t}+1}\right) \geqslant 1  \quad \text{a.s.}
\tag{S.61}
\end{align}

Using the definition of $\varphi(s)=\tilde{c} s^{1-\Tilde{\mu}}, \Tilde{\mu} \in[0,1), \tilde{c}>0$  in Theorem \ref{thm 3}, we can insert $\nabla \varphi(s)= \tilde{c}(1-\Tilde{\mu})s^{-\Tilde{\mu}} $ into (\ref{grad psi}) to deduce
\begin{align}\label{nabla psi}
\tilde{c}(1-\Tilde{\mu})(\Psi_{\tilde{t}+1}-\Psi^{*})^{-\Tilde{\mu}}dist\left(0, \partial \Psi_{\tilde{t}+1}\right) \geqslant 1  \quad \text{a.s.}, \nonumber\\
\left({\Psi}_{\tilde{t}+1}-\Psi^{*}\right)^{\Tilde{\mu}} \leq \tilde{c} (1-\Tilde{\mu}) dist\left(0, \partial \psi_{\tilde{t}}+1\right)  \quad \text{a.s.} \tag{S.62}
\end{align}
Using the expression for $\varphi(s)=\tilde{c} s^{1-\Tilde{\mu}}$, and the equation (\ref{equ 701}) again to obtain
\begin{align}\label{equ 68}
&\mathbb{E} \varphi\left(\Psi_{\tilde{t}+1}-\Psi^{*}\right)\nonumber\\
\leq &\gamma \mathbb{E} \bigg[\left\|\vx_{\tilde{t}+1}-\widetilde{\vx}\right\|+\left\|\vx_{\tilde{t}-1}-\widetilde{\vx}\right\| +\left\|\vx_{\tilde{t}+1}-\vx_{t}\right\|  +\left\|\vx_{\tilde{t}}-\tilde{\vx}\right\|+\left\|\vx_{\tilde{t}}-\vx_{\tilde{t}-1}\right\| \bigg]^{\frac{1-\Tilde{\mu}}{\Tilde{\mu}}}  \quad \text{a.s.}, \tag{S.63}
\end{align}
where  $\gamma = \tilde{c}\left[\tilde{c} (1-\Tilde{\mu}) \xi_{\max }\right]^{\frac{1-\Tilde{\mu}}{\Tilde{\mu}}}$.

Next,
we focus on the case where $\Tilde{\mu}\in\left(0, \frac{1}{2}\right]$ and $\frac{1-\Tilde{\mu}}{\Tilde{\mu}} \geq 1$. When $\tilde{t} \rightarrow +\infty$, the number under the $\frac{1-\Tilde{\mu}}{\Tilde{\mu}}$ root is infinitely small. And other cases can be proved similarly, which has been studied in \cite{bolte2014proximal}. In this case, it follows from
Equation (\ref{equ 68}) that 
\begin{align}\label{equ 69}
\mathbb{E}\varphi\left(\Psi_{\tilde{t}+1}-\Psi^{*}\right) \leq &\gamma \bigg[\mathbb{E}\left\|\vx_{\tilde{t}+1}-\widetilde{\vx}\right\|+\mathbb{E}\left\|\vx_{\tilde{t}-1}-\widetilde{\vx}\right\| +\mathbb{E}\left\|\vx_{\tilde{t}+1}-\vx_{\tilde{t}}\right\| +\mathbb{E}\left\|\vx_{\tilde{t}}-\tilde{\vx}\right\|\nonumber\\
&+ \mathbb{E}\|\vx_{\tilde{t}} -\vx_{\tilde{t}-1}\| \bigg] \quad \text{a.s.} \tag{S.64}
\end{align}

Setting $m$ in the equation (\ref{equ 62}) to be $+\infty$ with		
 $\lim_{m \rightarrow+\infty}\mathbb{E}\left\|\vx_{m}-\widetilde{\vx}\right\| = 0$,  $\lim_{m \rightarrow+\infty}\mathbb{E}\left\|\vx_{m+1}-\vx_{m}\right\|= 0  \quad \text{a.s.}$,  then we can see that
{
\begin{align}\label{equ 70}
& 2\sum_{t=\tilde{t}+1}^{+\infty}\mathbb{E}\left\|\vx_{t+1}-\vx_{t}\right\|+ \sum_{t=\tilde{t}+1}^{+\infty}\mathbb{E}\left\|\vx_{t-1}-\widetilde{\vx}\right\| + \mathbb{E}\left\|\vx_{\tilde{t}+1}-\widetilde{\vx} \right\| +2\mathbb{E}\left\|\vx_{\tilde{t}}-\widetilde{\vx} \right\|  \nonumber\\
\leqslant & \mathbb{E}\left\|\vx_{\tilde{t}+1}-\vx_{\tilde{t}}\right\| + (2 \sqrt{\xi_{max}} \sqrt{\frac{2}{\tau}} )^2\mathbb{E}\varphi\left(\Psi_{\tilde{t}+1}-\Psi^{*}\right)< + \infty  \quad \text{a.s.}
\tag{S.65}
\end{align}}

Now set
{\setlength{\abovedisplayskip}{2pt}
\setlength{\belowdisplayskip}{2pt}
\begin{align}
\Delta_{t}^{1}&:=\sum_{i=t}^{+\infty}\mathbb{E}\left\|\vx_{i+1}-\vx_{i}\right\|,\nonumber\\
\Delta_{t}^{2}&:=\sum_{i=t}^{+\infty}\mathbb{E}\left\|\vx_{i}-\widetilde{\vx}\right\|, 
\tag{S.66}
\label{delta 12}
\end{align}}
and $L_{\varphi} = (2 \sqrt{\xi_{max}} \sqrt{\frac{2}{\tau}} )^2$, it follows from
Equations (\ref{equ 69}) and (\ref{equ 70}) that
{	\setlength{\abovedisplayskip}{2pt}
\setlength{\belowdisplayskip}{2pt}
\begin{align}
&2\Delta_{\tilde{t}+1}^{1} + \Delta_{\tilde{t}}^{2} +2 \Delta_{\tilde{t}}^{2}- 2 \Delta_{\tilde{t}+1}^{2}+\Delta_{\tilde{t}+1}^{2}-\Delta_{\tilde{t}+2}^{2} \nonumber\\
\leq & \Delta_{\tilde{t}}^{1}-\Delta_{\tilde{t}+1}^{1}+ L_{\varphi} \gamma \big[\Delta_{\tilde{t}+1}^{2}-\Delta_{\tilde{t}+2}^{2}
+\Delta_{\tilde{t}}^{1}-\Delta_{\tilde{t}+1}^{1}+\Delta_{\tilde{\tilde{t}}}^{2}-\Delta_{\tilde{t}+1}^{2}
+\Delta_{\tilde{t}-1}^{1}-\Delta_{\tilde{t}}^{1}+ \Delta_{\tilde{t}-1}^{2}-\Delta_{\tilde{t}}^{2}\big]  \quad \text{a.s.} \nonumber
\end{align}}

To simplify this expression, let's replace the label $\tilde{t}$ with $t$ and  set
\begin{align}
a_{t} := \Delta_{\tilde{t}}^{1},~	b_{t} := \Delta_{\tilde{t}}^{2},
\tag{S.67} \label{a bt}
\end{align}
the above inequality can rewritten as the following form
{	
\begin{align}\label{at and bt}
0 & \leq -(L_{\varphi}\gamma+3)a_{t+1} + a_{t} + L_{\varphi }\gamma a_{t-1} + L_{\varphi} \gamma  b_{t-1} -3b_{t} +b_{t+1} + (1-L_{\varphi }\gamma )b_{t+2}  \quad \text{a.s.}
\tag{S.68}
\end{align}}

Let us introduce some constants $A, B, C, D, E, F,$ and $ G$ satisfied  the following conditions, which is the key trick  to driving the linear convergence rate
\setlength{\abovedisplayskip}{10pt}
\setlength{\belowdisplayskip}{10pt}
\begin{equation} \label{equ abc}
\left\{\begin{array}{l}
A=L_{\varphi} \gamma, \\
B-A C=1, \\
-B C=-(L_{\varphi}\gamma+3), \\
B=1+ L_{\varphi}\gamma \times C, \\
D=L_{\varphi}\gamma, \\
E-G D=-3, \\
F-G E=1,  \\
-G F = 1- L_{\varphi}\gamma. \\
\end{array}\right.
\tag{S.69}
\end{equation}
{Then} (\ref{at and bt}) can be rewritten as
{	
\begin{align}\label{equ 755}
0 \leq & -BC a_{t+1} + (B-AC)a_{t} +Aa_{t-1} + D b_{t-1} + (E-GD)b_t +(F-GE)b_{t+1} - GF b_{t+2}  \quad \text{a.s.} \tag{S.70}
\end{align}}
We can rearrange the terms in the above inequality (\ref{equ 755}) to obtain
{	
\begin{align}
C\left(A a_{t}+B a_{t+1}\right)+G\left(D b_{t}+E b_{t+1}+F b_{t+2}\right)
\leqslant & (A a_{t-1}+B a_{t})+\left(D b_{t-1}+E b_{t}+F b_{t+1}\right)  \quad \text{a.s.}
\tag{S.71}
 \label{equ 733}
\end{align}}
Let $I=\min \{C, G\} >1$, from (\ref{equ 733}), we have
{	
\begin{align}
I\left[\left(A a_{t}+B a_{t+1}\right)+\left(D b_{t}+E b_{t+1}+F b_{t+2}\right)\right]
\leqslant & A a_{t-1}+B a_{t}+\left(D b_{t-1}+E b_{t}+F b_{t+1}\right)  \quad \text{a.s.} 
\tag{S.72}
\label{equ 744}
\end{align}}
Denote
{	\setlength{\abovedisplayskip}{2pt}
\setlength{\belowdisplayskip}{2pt}
\begin{equation}
N_{t}:=A  a_{t-1}+B a_{t}+\left(D b_{t-1}+E b_{t}+F b_{t+1}\right),
\tag{S.73} \label{Nt}
\end{equation}}
and insert the symbol $N_{t}$ and $N_{t-1}$ into the inequality (\ref{equ 744}) we get the desired result
{	
\begin{equation}
\frac{N_{t+1}}{N_{t}} \leq \frac{1}{I}  \quad \text{a.s.}
\tag{S.74} 
 \label{Q linear}
\end{equation}}

Then we prove the existence of the parameter. First, we obtain the part of the solution in the equation (\ref{equ abc}):
{	\setlength{\abovedisplayskip}{10pt}
\setlength{\belowdisplayskip}{10pt}
$$
\left\{\begin{array}{l}
L_{\varphi}\gamma C^{2 } +C -(L_{\varphi} \gamma+3)=0, \\
B=L_{\varphi} \gamma C +1, \\
C=\frac{-1 + \sqrt{1+4 L_{\varphi} \gamma (3+L_{\varphi}\gamma)}}{2 L_{\varphi} \gamma} >1,0<\frac{1}{C}<1. \end{array}\right.
$$}

Also, we have the rest part of the solution in the equation (\ref{equ abc}):
{	\setlength{\abovedisplayskip}{10pt}
\setlength{\belowdisplayskip}{10pt}
\begin{equation} \label{equ defgg}
\left\{\begin{array}{l}
D=L_{\varphi}\gamma, \\
E-G D=-3, \\
F-G E=1,  \\
1 - L_{\varphi}\gamma)=-G F. \\
\end{array}\right.
\tag{S.75} 
\end{equation}}

Using the equation (\ref{equ defgg}), we have{	\setlength{\abovedisplayskip}{10pt}
\setlength{\belowdisplayskip}{10pt}
\begin{equation} \label{equ rewrite defg}
\left\{\begin{array}{l}
E = L_{\varphi}\gamma G -3, \\
F = GE +1, \\
L_{\varphi}\gamma  G^{3}- 3G^{2}+G+(1- L_{\varphi}\gamma)=0. \\
\end{array}\right.\nonumber
\end{equation}}

Set $H(G) = L_{\varphi}\gamma  G^{3}- 3G^{2}+G+(1- L_{\varphi}\gamma)$, $H(1)<0, H(+\infty)>0$. Thus we can choose the root $G>1$ of $H(G) $. Since the parameters $G$ and $C$ defined in (\ref{equ abc}) are both greater than $1$, we can obtain that $I$  satisfies this condition $I=\min \{C, G\} >1$. The existence of these parameters defined in (\ref{equ abc}) and $I$ are proved.

Second,
the above shows that the sequence $\left\{ N_{t} \right\}$ converges a.s. to zero Q-linearly\footnote{For the sequence {$\left\{x_n\right\}_{n \in \mathbb{N}}$ with  		
$
\lim\limits_{n \rightarrow \infty} x_n=x^{*},
$
if}
$$
\lim _{n \rightarrow \infty} \frac{\left\|x_{n+1}-x^{*}\right\|}{\left\|x_n-x^{*}\right\|}\leq \gamma,
$$
where $0<\gamma<1$,  {then} the sequence $\left\{x_n\right\}_{n \in \mathbb{N}}$is said to converge to $x^{*}$ Q-linearly, and the constant $\gamma$ is called the rate of (linear) convergence.}. As a result, we can use the triangle inequality with the notations in (\ref{delta 12}) and (\ref{a bt}) to yield:
\begin{align}\label{triangle}
A\mathbb{E}\left\|\vx^*-\vx_{t}\right\|= &A \mathbb{E}\left\|\sum_{i=t}^{+\infty}(\vx_{i+1}-\vx_{i})\right\|, \nonumber\\
\leq & A\sum_{i=t}^{+\infty}\mathbb{E}\left\|\vx_{i+1}-\vx_{i}\right\|, \nonumber\\
\leq &N_t  \quad \text{a.s.},
\tag{S.76} 
\end{align}
and it is sufficient to say the sequence $\left\{ \vx_t\right\}$ converges a.s. to $\vx^*$ R-linearly from the definition of R-linear convergence in Section \ref{sec:intro}. Furthermore, from the formulas (\ref{y-bound}) and (\ref{upp1}) we can see the sequences $\left\{ \vy_t\right\}$ and $\left\{ \vlambda_t\right\}$ can be controlled by the sequence $\left\{ \vx_t\right\}$ and converge a.s. to $\vy^*$ and $\vlambda^*$ R-linearly, respectively. Combing the R-linear convergences of the sequences $\left\{ \vx_t\right\}$, $\left\{ \vy_t\right\}$, and $\left\{ \vlambda_{ t} \right\}$ and the definition of R-linear convergence in Section \ref{sec:intro}, we can derive the desired result: the sequence
$\left\{ \left(\vx_{t}, \vy_{t}, \vlambda_{t}\right) \right\}$ converges a.s. to $\left(\vx^{*}, \vy^{*}, \vlambda^{*}\right)$ R-linearly with a existing sequence $\left\{ \hat{c}  \xi_{3}^{t} \right\} $ as follows:

$$
\mathbb{E}\left\|\left(\vx_{t}, \vy_{t}, \vlambda_{t}\right)-\left(\vx^{*}, \vy^{*}, \vlambda^{*}\right)\right\| \leq \hat{c}  \xi_{3}^{t}  \quad \text{a.s.}
$$
and
$$\left\|\left(\vx_{t}, \vy_{t}, \vlambda_{t}\right)-\left(\vx^{*}, \vy^{*}, \vlambda^{*}\right)\right\| \leq \hat{c}  \xi_{3}^{t}  \quad \text{a.s.}
$$
{This} completes the whole proof.

\hfill $\blacksquare$

\end{document}